\documentclass[a4paper,10pt]{article}
 \usepackage[utf8x]{inputenc}
 \usepackage[T1]{fontenc}
 \usepackage[normalem]{ulem}
 \usepackage[english]{babel}
 \usepackage[toc,page]{appendix} 
 \usepackage{amsfonts}
 \usepackage{graphicx}
  \usepackage{mathenv}
  \usepackage{amsmath}
  \usepackage{slashbox}
    \usepackage{url}
  \usepackage{appendix}
\usepackage{amsthm}
 \usepackage{pgf,tikz}
 \usetikzlibrary{arrows}
\newtheorem{theorem}{Theorem}[section]
\newtheorem*{thma}{Theorem (Maillard \cite{maillard2013number})}
\newtheorem*{thmb}{Theorem (Harris et al. \cite{harris2006further})}
\newtheorem*{thmc}{Theorem (Flajolet, Corollary VI.1 of \cite{flajolet2009analytic})}
 \newtheorem{lemma}[theorem]{Lemma}

  \newtheorem{proposition}[theorem]{Proposition}
  \newtheorem{corollary}[theorem]{Corollary}
  \title{Number of particles absorbed in a BBM on the extinction event}
   \PrerenderUnicode{é}
   \author{ Pierre-Antoine Corre  \footnote{%
Laboratoire de Probabilités et Modèles Aléatoires, CNRS UMR 7599, Université Pierre et Marie Curie, Paris 6, Case courrier 188, 4 place Jussieu 75252 Paris Cedex 05, email: \texttt{pierre-antoine.corre@upmc.fr}}}
\begin{document}
\definecolor{ffqqqq}{rgb}{1,0,0}
\definecolor{qqqqff}{rgb}{0,0,1}
\definecolor{xdxdff}{rgb}{0.49,0.49,1}
\maketitle
\begin{abstract}
 We consider a branching Brownian motion which starts from $0$ with drift $\mu \in \mathbb{R}$ and we focus on the number $Z_x$ of particles killed at $-x$, where $x>0$. Let us call $\mu_0$ the critical drift such that there is a positive probability of survival if and only if $\mu>-\mu_0$. Maillard \cite{maillard2013number} and Berestycki et al. \cite{berestycki2015branching} have study $Z_x$ in the case $\mu \leq -\mu_0$ and $\mu\geq \mu_0$ respectively. We complete the picture by considering the case where $\mu>-\mu_0$ on the extinction event.  More precisely we study the asymptotic of $q_i(x):=\mathbb{P}\left(Z_x=i,\zeta_x<\infty\right)$. We show that the radius of convergence $R(\mu)$ of the corresponding power series increases as $\mu$ increases, up until $\mu=\mu_c\in [-\mu_0,+\infty]$ after which it is constant. We also give a necessary and sufficient condition for $\mu_c<+\infty$. In addition, finer asymptotics are also obtained, which highlight three different regimes depending on $\mu<\mu_c$, $\mu=\mu_c$ or $\mu>\mu_c$.
\end{abstract}
\section{Introduction and main results}
We consider a branching Brownian motion which starts from $0$ with drift $\mu \in \mathbb{R}$, branching rate $\beta>0$, and reproduction law $L$. Let us recall the definition of such a process: a particle starts from $0$, lives during an exponential $\beta$ random time and moves as a Brownian motion with drift $\mu$. $\mathbb{N}$ denotes as usual the set $\lbrace0,1,2,\cdot\cdot\cdot\rbrace$. When a particle dies, it gives birth to a random number $L\in \mathbb{N}\setminus \lbrace 1 \rbrace$ of independent branching Brownian motions started at the position where it dies.  We denote by $G$ the generating function of $L$, that is 
\begin{equation}\label{defG}
G(s)=\mathbb{E}\left(s^L\right)=\sum_{i=0}^{\infty} p_i s^i,
\end{equation} 
where $p_i=\mathbb{P}\left(L=i\right)$ and we call $R_G$ the radius of convergence of $G$. In this article, we will always assume that:
\begin{equation}
m=\mathbb{E}(L)\in (1,+\infty)\mbox{ and }\mu >-\mu_0\mbox{ or }m=+\infty.
\end{equation}
where $\mu_0=\sqrt{2 \beta (m-1)}$, which corresponds, when $m<\infty$, to the speed of the maximum $M_t$ of a branching Brownian motion without drift in the sense that $ M_t/t \displaystyle \underset{ +\infty}{\rightarrow} \mu_0$ almost surely on the survival event. 

In our model, we kill the particles when they first hit the position $-x,$ $x>0$. We call $\zeta_x$ the extinction time of the process, that is the first time when all particles have been killed. Let us define the extinction probability  
\begin{equation}
Q(x,\mu):=\mathbb{P}\left(\zeta_x < + \infty \right)
\end{equation} 
(or often simply $Q(x)$ when no confusion can arise). We will, throughout this paper, use the classical notation $\left(\mathcal{T},\mathcal{F},(\mathcal{F}_t),\mathbb{P}\right)$ to denote the filtered probability space on which the Branching Brownian motion evolves, see for instance $\cite{hardy2009spine}$ for more details. 

 Our purpose is to study the number $Z_x$ of particles killed at $-x$ for a branching Brownian motion on the extinction event.  We can distinguish 3 different cases according to the drift value when $1<m<+\infty$.
\begin{enumerate} 
\item If $\mu \leq -\mu_0$ there is extinction almost-surely and $Z_x<\infty$ a.s.
\item If $|\mu| < \mu_0$ then the survival probability is non-zero. The number of particles is almost-surely finite if extinction occurs and is almost-surely infinite otherwise.
\item If $\mu \geq \mu_0$  then the survival probability is non-zero. The number of absorbed particles is almost-surely finite, whether extinction occurs or not.
\end{enumerate} 
When $m=+\infty$, we can consider that we are in the second case. The first case has been studied by Maillard \cite{maillard2013number} and the third case by Berestycki et al. \cite{berestycki2015branching}. 
Here, we consider both case 2 and 3 (that is $\mu>-\mu_0$), on the extinction event. We can point out that since $Z_x=\infty$ a.s. on the survival event in case 2, the restriction to the extinction event in this provides a whole description of $Z_x$. 

Initially, in the context of branching random walk with absorption on a barrier, the issue of the total number of particles $Y$ that have lived before extinction on a barrier had been studied by Aldous \cite{Aldous}. He conjectured that there exist $K,b>1$ such that in the critical case (which is the analogue of $\mu=-\mu_0$ for the branching Brownian motion) we have that $\mathbb{P}(Y>n) \sim_{n \rightarrow + \infty} K/n^{b}$ and that in the sub-critical case (which is the analogue of $\mu<-\mu_0$) we have that $\mathbb{E}(Y)<+\infty$ and $\mathbb{E}(Y \log Y)=+\infty$. This problem has been solved by Addario-Berry et al. in \cite{addario2011total} and Aïdekon et al. in \cite{aidekon2013precise}  refined their results. Maillard has given a very precise description of the number $Z_x$ of particles which are killed on the barrier $-x$, $x>0$ when $\mu\leq -\mu_0$ for the branching Brownian motion. More precisely, he showed the following result. Fix $\delta$ the span of $L$, that is the greatest positive integer such that the support of $L-1$ is on $\delta \mathbb{Z}$. Define $\lambda_1:=-\mu+\sqrt{\mu^2-2\beta}$, $\lambda_2:=-\mu-\sqrt{\mu^2-2\beta}$, and $d:=\lambda_1/\lambda_2$. 
\begin{thma}
Assume that $\mathbb{E}\left[L\log^2 L\right]<+\infty$. If $\mu=-\mu_0$, then 
\begin{equation}
\mathbb{P}\left(Z_x>n\right)\underset{n\rightarrow+\infty}{\sim} \frac{\mu_0xe^{\mu_0 x}}{n(\log n)^2}.
\end{equation}
Assume now that $R_G$, the radius of convergence of the generating function of $L$, is greater than $1$. We then have that: 
\begin{itemize}
\item If $\mu=-\mu_0$:
\begin{equation}
\mathbb{P}\left(Z_x=\delta n+1\right)\underset{n\rightarrow+\infty}{\sim} \frac{\mu_0xe^{\mu_0 x}}{\delta n^2(\log n)^2}.
\end{equation}
\item If $\mu<-\mu_0$, there exists $K>0$ such that:
\begin{equation}
\mathbb{P}\left(Z_x=\delta n+1\right)\underset{n\rightarrow+\infty}{\sim} K\frac{e^{\lambda_1 x}-e^{\lambda_2 x}}{n^{d+1}}.
\end{equation}
\end{itemize}
\end{thma}

To prove this theorem Maillard introduces the generating function of $Z_x$ defined for $s \in \mathbb{R}^+$ by: 
\begin{equation}
F_x(s)=\mathbb{E}\left(s^{Z_x}\right).\label{defbigF}
\end{equation} 
Since we want to work on the extinction event we will rather work with:
\begin{equation}
f_x(s):=\mathbb{E}\left(s^{Z_x}\mathbf{1}_{\lbrace \zeta_x<\infty \rbrace}\right)=\sum_{i=0}^{\infty} q_i(x)s^i, \quad s\in \mathbb{R}^+,
\end{equation}
where 
\begin{equation} q_i(x)=\mathbb{P}\left(Z_x=i,\zeta_x<\infty\right),
\end{equation}
to prove an analogous theorem. Note that $F_x(s)$ and  $f_x(s)$ coincide in two cases. The first one, dealt with by Maillard, happens when $m<\infty$ and $\mu \leq -\mu_0$ because the process becomes extinct almost surely. The second case happens for $s\in[0,1)$ when $m=+\infty$ or when $m<\infty$ and $|\mu| < \mu_0$, since for this range of $\mu$ the event $\lbrace\zeta_x=\infty\rbrace$ is almost surely equal to $\lbrace Z_x=\infty\rbrace$. The reason for which we chose to consider $f_x$ instead $F_x$ is that, for $s>1$, and $|\mu| < \mu_0$, $F_x(s)$ is infinite (because $\lbrace Z_x=\infty\rbrace$ happens with non-zero probability). Even in the case $\mu \geq \mu_0$, our situation is clearly different from that in \cite{berestycki2015branching}, since we restrict to extinction and only the binary branching mechanism is considered in \cite{berestycki2015branching}. 

As a first step we will focus on the radius of convergence of $f_x$ denoted by $R(\mu)$ and we will show that it depends on $\mu$ but not on $x$, which justifies the notation $R(\mu)$. The quantity $R(\mu)$ gives us a first information on $Z_x$, in particular via the Cauchy-Hadamard Theorem (see for instance \cite{lang2013complex}) which tells us that:
\begin{equation}\label{HCth}
\limsup_{n \rightarrow + \infty} \left| q_n(x)\right|^{\frac{1}{n}}=\frac{1}{R(\mu)}.
\end{equation}
A key tool in the present work is $Q$. It satisfies the KPP travelling wave equation and is its unique solutions under some boundary conditions. This result, which is stated in $\cite{harris2006further}$ in the binary case ($L \equiv 2$), is given in the following theorem. Let $q$ be the probability of extinction without killing on the barrier or equivalently the smallest non-negative fixed point of $G$ (defined in \eqref{defG}).
\begin{thmb}\label{thKPP}

$Q$ is the unique solution in $\mathcal{C}(\mathbb{R}^+,[0,1])$ of the equation:
\begin{equation}\label{KKPx}
\frac{1}{2}y^{\prime \prime}(x)+\mu y^{\prime}(x)+\beta\left(G(y(x))-y(x)\right)=0,
\end{equation}
with boundary conditions:
\begin{equation}\label{KKPxinitial}
y(0)=1, y(\infty)=q,
\end{equation}
when $m=+\infty$ or $\mu>-\mu_0$. There is no such solutions when $m<+\infty$ and $\mu\leq-\mu_0$.
\end{thmb}
The arguments presented in $\cite{harris2006further}$ work without modification in the general case except one. Indeed, the non-triviality of $Q$ is proved when $\mu > -\mu_0$ by using the convergence of the additive martingale to a non-trivial limit (see for instance \cite{hardy2009spine}). But this convergence requires the condition $\mathbb{E}\left(L\log L\right)<\infty$. Maillard gives a proof of the non-triviality of $Q$ in the supercritical case, without assumptions on $L$, which proves that this theorem is always true. Note finally that $x \mapsto F_x(s)$ and $x \mapsto f_x(s)$ also satisfy $\eqref{KKPx}$, but only $x \mapsto Q(x)=f_{x}(1)$ satisfies the boundary conditions \eqref{KKPxinitial}.

As a solution of  \eqref{KKPx}, we can extend $Q$ to an open interval containing $\mathbb{R}_+^*$ but also to a complex domain. The following result is a reformulation in our setting of two classical theorems (Theorem 3.1 of Chapter II of \cite{coddington1955theory} and Section 12.1 of \cite{ince1927ordinary}) applied to $\eqref{deff}$, which gives such extensions.  We define a neighbourhood of a point by a simply connected open which contains this point. 
\begin{proposition}\label{extension}
There exists a maximal open interval $I$ such that we can extend $Q$ on $I$ as a solution of \eqref{KKPx} and such that  $Q(x) \in (-R_G,R_G), \forall x \in I $. This extension is unique. Let us define $x_l =\inf I$, we further have that if $x_l>-\infty$, then:
\begin{equation} \label{boundxl}
\lim_{x \rightarrow x_l^+} \mid Q(x) \mid=R_G \mbox{ or } \limsup_{x \rightarrow x_l^+} \mid Q^{\prime}(x) \mid=+\infty.
\end{equation}
Moreover for each $x \in I$, $Q$ admit an analytic continuation on a neighbourhood of $x$ (in the complex sense). 
\end{proposition}
Since the extension described in Proposition \ref{extension} is unique, we will make a slight abuse of notation and write $Q$ to denote this extension. If $x_l>-\infty$ either $Q$ cannot be extended analytically left of $x_l$ or such an extension would exit $(-R_G,R_G)$. \\

Finally, we give the connection between $f_x$ and $Q$. The branching property yields: 
\begin{equation}\label{branchprop}
f_{x+y}(s)=f_y(f_x(s)), \forall (x,y,s) \in {\left(\mathbb{R}^+\right)}^2 \times \mathbb{R}^+,
\end{equation}
where the two sides can possibly be equal to $+\infty$, see Maillard \cite{maillard2013number} for the analogous property for $F$. Now consider $J$ the maximal open interval included in $I$ which contains $(0,+\infty)$ such that $Q$ is decreasing on $J$. The function $Q$ is thus invertible on $J$.  Fix $x_0(\mu):=\inf J$. By another slight abuse of notation, we define $Q(x_0(\mu))$ as the right-limit of $Q$ when $x$ goes to $x_0$. Note that this limit exists because $Q$ is decreasing and bounded on $J$. For $x\in\mathbb{R}^+$, since $f_x(1)=Q(x)$, we can derive from \eqref{branchprop} that:
\begin{equation}\label{linkomegQ}
f_x(s)=Q\left(Q^{-1}(s)+x\right), \forall q < s\leq R(\mu)\wedge Q(x_0(\mu)).
\end{equation}
We choose in the previous equation $q < s\leq R(\mu)\wedge Q(x_0(\mu))$ to ensure that the two terms of the quality are well-defined. Actually, the following description of $R(\mu)$ shows us that we can chose $s\in (q , R(\mu)]$.
\begin{theorem} \label{thcaractRwrtmu} 
Let $\mu\in\mathbb{R}$, 
\begin{equation}
R(\mu)=Q(x_0(\mu),\mu).
\end{equation}
\end{theorem}
We will write $x_0$ rather $x_0(\mu)$ when no confusions can arise. With the help of Theorem \ref{thcaractRwrtmu}, we can state the behaviour of $R(\mu)$ with respect to $\mu$. 
\begin{theorem}\label{thbehaviourRmu}
The radius of convergence $R(\mu)$ is a non-decreasing continuous function of $\mu$ such that:
\begin{equation}
\lim_{\mu \rightarrow -\mu_0}  R(\mu)=1
\end{equation}
and
\begin{equation}
\lim_{\mu \rightarrow +\infty}  R(\mu)=R_G.
\end{equation}
\end{theorem}
In particular, if $R_G=1$ then $R(\mu)=1, \; \forall \mu \in \mathbb{R}$. \\
We want now to have a more accurate result than \eqref{HCth} concerning the asymptotic behaviour of $q_n(x)$ when $n$ tends to infinity. For this, we will first find a good domain (we will say what good means later) on which $f_x$ is analytical. We will next study the behaviour of $f_x$  near $R(\mu)$ (in the real or complex sense, as appropriate). More precisely, the goal is to obtain a classical function equivalent of $f_x$ in a  neighbourhood of $R(\mu)$. If these two conditions are satisfied, we can give an exact equivalent of $q_n(x)$ when $n$ tends to infinity thanks to analytical methods. Fortunately, this is the case when $R(\mu)<R_G$. A natural question is then to know whether $R(\mu)$ reaches $R_G$ for a finite $\mu$. Let us define $\mu_c$ as:
\begin{equation}\label{defmuc}
 \mu_c=\inf \lbrace \mu \in \mathbb{R}, R(\mu)=R_G \rbrace.
 \end{equation}
 In particular, we have $R(\mu)<R_G, \forall \mu \in \mathbb{R}$,  if and only if $\mu_c=+\infty$. The following figure shows what happens when $\mu<\mu_c$ or when $\mu>\mu_c$.\\

\begin{figure}[!h]
\centering
\includegraphics{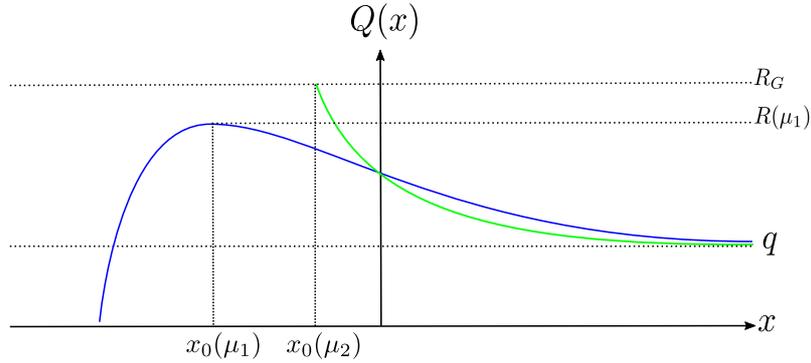}

 \caption{ Let $(\mu_1,\mu_2)\in \mathbb{R}^2$ such that $-\mu_0<\mu_1<\mu_c<\mu_2$. We represent in blue $Q(\cdot,\mu_1)$ and in green $Q(\cdot,\mu_2)$.}
\end{figure}

  The next theorem gives us a criterion to know whether $R_G$ is reached or not for a finite $\mu$.
\begin{theorem}\label{threachornot}
\begin{equation}
\mu_c<\infty \Leftrightarrow \int_{0}^{R_G} G(s) \mbox{d}s<+\infty.
\end{equation}
\end{theorem}
Note that the case $R_G=+\infty$ is included in the case $\int_{0}^{R_G} G(s) \mbox{d}s=+\infty$ of Theorem \ref{threachornot}. We can now give an asymptotic equivalent to $q_n(x)$ when $n$ tends to infinity for $\mu<\mu_c$. We recall that $\delta$ is the span of $G$ and $x_0$ is defined above as the position of the local maximum of $Q$ the nearest to 0.
\begin{theorem}\label{asymptotic}
When $\mu<\mu_c$, for $x>0$ we have:
\begin{equation}q_{\delta i +1}(x) \underset{i\rightarrow+\infty}{\sim}  \frac{-Q^{\prime}(x_0(\mu)+x)}{2{R(\mu)}^{\delta i+\frac{1}{2}}\sqrt{\delta\beta(G(R(\mu))-R(\mu))i^3\pi}}.\label{asymptoticcontenu}
\end{equation}
\end{theorem}
When $\mu\geq\mu_c$ we cannot use the same techniques. We will explain why in the last section, but roughly speaking, the reason is that in the general case when $\mu\geq\mu_c$, we cannot extend $f_x$ on a complex domain big enough to apply Flajolet singularity analysis \cite{flajolet2009analytic}, which is the key to Theorem \ref{asymptotic}. Nevertheless, we can obtain some results for no too restrictive hypothesis by applying a classical Tauberian theorem. We present some particular examples in the last section and highlight a change of regime when $\mu=\mu_c$ and when $\mu>\mu_c$.\\

The paper is organised as follows. Section 2 concerns the extinction probability. Some important results are recalled. In Section 3, we will give the main properties of $R$ and prove Theorems $\ref{thbehaviourRmu}$ and $\ref{threachornot}$. Section 4 is devoted to the proof with analytical methods of Theorem \ref{asymptotic}. Finally, in the last section, we will consider the case $\mu \geq \mu_c$.

\section{First results on the extinction probability}
In this section, we give the main properties on $Q$ which will allow us to determine the radius of convergence of $f_x$ in the next section. Since we want to use phase portraits techniques, we consider:
\begin{equation}\label{defX}
X(x,\mu):=\left( Q(x,\mu),Q^{\prime}(x,\mu)\right)\in \mathbb{R}^2.
\end{equation}
 Once again, we will often write $X(x)$ instead of $X(x,\mu)$. We can rewrite the KPP travelling wave equation satisfied by $Q$ as:
\begin{equation}\label{deff}
X^{\prime}=\Gamma(X,\mu),\mbox{ where } \Gamma((x,y),\mu)=(y,-2\mu y-2\beta(G(x)-x)).
\end{equation}

We will need the precise behaviour of $Q$ and $Q^{\prime}$ in the neighbourhood of $+\infty$. In \cite{harris2006further}, Harris et al. give the asymptotic equivalent of $Q$ in the binary case, we will state here a more precise version of this result in the general case. \\

\begin{theorem}\label{thequiv}

If $m=+\infty$ or $m<+\infty$ and $\mu>-\mu_0$, there exists $k>0$ such that:
\begin{equation}\label{equivalent}
Q(x)=q+ke^{-\tilde{\lambda} x} + \underset{x\rightarrow+\infty}{o} (e^{-\tilde{\lambda} x})
\end{equation}
and
\begin{equation}\label{equivalent2}
Q^{\prime}(x) \underset{x\rightarrow+\infty}{\sim}-k\tilde{\lambda} e^{-\tilde{\lambda} x},
\end{equation}
where $\tilde{\lambda}=\sqrt{2\beta(1-G^{\prime}(q))+\mu^2}+\mu$. 
\end{theorem}
If $q=0$, $G^{\prime}(q)=p_1=0$ and thus $\lambda:=\sqrt{2\beta+\mu^2}+\mu=\tilde{\lambda}$, which is exactly the result in \cite{harris2006further}.
Note that this result could be refined by showing that $Q$ is a Dirichlet series as it is done for another travelling-wave in \cite{berestycki2015branching}.\\

 The following lemma reformulates the KPP equation in two ways. The first one is obtained by stopping the process at the first branching time. The second one is obtained by multiplying all terms of the KPP travelling-wave equation by $e^{2\mu x}$, and by integrating this equation from $0$ to $x$.
\begin{lemma} \label{thKPPtempsarretx}
Let $\alpha=\sqrt{2\beta+\mu^2}$ and $\lambda=\alpha+\mu$. For $x\geq 0$, we have:
\begin{equation}\label{KKPtempsarretx}
Q(x) = e^{-\lambda x}+\frac{\beta}{\alpha} \int_0^{+\infty}e^{-\mu x}e^{\mu y}G(Q(y))\left[e^{-\alpha |y-x|}-e^{-\alpha (y+x)}\right] \mathrm{d}y.
\end{equation}
Moreover, for $x>x_l$, we have:
\begin{equation}\label{Qprimeequ}
Q^{\prime}(x)=\left(Q^{\prime}(0)-2\beta\int_0^x e^{2 \mu y }\left(G(Q(y))-Q(y)\right)\mathrm{d}y \right)e^{-2 \mu x}.
\end{equation}
\end{lemma}
\begin{proof}
We prove \eqref{KKPtempsarretx} only. Let $x \geq 0$. We decompose the event $\lbrace \zeta_x<+\infty \rbrace$ on two sub-events:
\begin{align}
Q(x)&= \mathbb{E}\left(\mathbf{1}_{\lbrace \zeta_x<+\infty\rbrace}\mathbf{1}_{\lbrace T_1\geq\zeta_ x\rbrace} \right)+ \mathbb{E}\left( \mathbf{1}_{\lbrace \zeta_x<+\infty\rbrace}\mathbf{1}_{\lbrace T_1<\zeta_x \rbrace} \right) \nonumber \\
&= \mathbb{P}\left(T_1\geq\zeta_ x\right)+ \mathbb{E}\left( \mathbf{1}_{\lbrace \zeta_x<+\infty\rbrace}\mathbf{1}_{\lbrace T_1<\zeta_x \rbrace} \right), \label{distfirstimeQ}
\end{align}
where $T_1$ is the time of first split. The first term in the right-hand side of $\eqref{distfirstimeQ}$ is the probability that a Brownian motion with drift $\mu$ starting from $x$ reaches $0$ before a exponential time with parameter $\beta$. By formula 1.1.2 p.250 of \cite{borodin2002handbook}, we thus have:
\begin{equation}\label{divapresreach}
\mathbb{P}\left(T_1\geq\zeta_ x\right)=e^{-\lambda x}.
\end{equation}  
We now look the second term. It is the probability that a Brownian motion with drift $\mu$ splits before reaching $0$ and that each process starting from its children becomes extinct. Consider $(K^x_s)$ a Brownian motion with drift $\mu$ starting from $x$ and killed at $0$. Using the Markov property and the independence between the Brownian motions, the first split time and the number of children, we have that:
\begin{align}
\mathbb{E}\left( \mathbf{1}_{\lbrace \zeta_x<+\infty\rbrace}\mathbf{1}_{\lbrace T_1<\zeta_x \rbrace}\right) &=\mathbb{E}\left( \mathbf{1}_{\lbrace K^x_{T_1}>0\rbrace}Q(K^x_{T_1})^L\right) \nonumber \\
&=\mathbb{E}\left( \mathbf{1}_{\lbrace K^x_{T_1}>0\rbrace}G\left(Q\left(K^x_{T_1}\right)\right)\right) \nonumber \\
&= \int_0^{+\infty}G(Q(y))\mathbb{P}\left(K^x_{T_1}\in \mathrm{d}y\right). \label{DKX1}
\end{align}
We can derive from 1.0.5 and 1.1.6 p.250-251 of \cite{borodin2002handbook} that: 
\begin{equation}
\mathbb{P}\left(K^x_{T_1}\in \mathrm{d}y\right)=\frac{\beta}{\alpha} e^{-\mu x}e^{\mu y}\left[e^{-\alpha |y-x|}-e^{-\alpha (y+x)}\right] \mathrm{d}y\label{dKX}
\end{equation}
and thus by plugging \eqref{dKX} into \eqref{DKX1} we get the desired result.
\end{proof}
We will see in Proposition \eqref{premdist} that $x_0:=\inf J$ is finite. Therefore, Equation \eqref{Qprimeequ} provides the existence in $\mathbb{R}\cup \lbrace -\infty \rbrace$ of the right-limit $Q^{\prime}$ as $x$ tends to $x_0$. Actually, this limit is in $\mathbb{R}$ (see the remark just after Lemma \ref{continuitymu0}). As above for $Q$, we denote by $Q^{\prime}(x_0)$ this limit, when it is finite. \\

 To prove Theorem \ref{thequiv}, we now give a bound of $Q$ in the following lemma. Although stated for the binary case in  Lemma 15 of \cite{harris2006further}, the result holds more generally when we just suppose $G(0)=0$ (which is equivalent to $p_0=0$ or $q=0$). 
\begin{lemma}\label{majQinf}
If $G(0)=0$, $\mu > -\mu_0$ then for all $0<y<x$:
\begin{equation}
Q(x) \leq (Q(y) e^{\rho y})e^{-\rho x},
\end{equation}
where $\rho=\sqrt{\mu^2+2\beta(1-Q(y))}+\mu$.
\end{lemma} 
In particular, this lemma tells us that for any $\epsilon>0$ there exist $x_1,k>0$ such that for any $x>x_1$:
\begin{equation}
Q(x)\leq ke^{-(\lambda-\epsilon)x},
\end{equation}
where $\lambda=\sqrt{\mu^2+2\beta}+\mu.$ The proof of Lemma \ref{majQinf} is identical  to that of  \cite{harris2006further} except that we are not in the binary case. Therefore, if  $Y_t$ is a Brownian motion with drift $\mu$ starting from $0$ and $\tau_z:=\inf \lbrace t : Y_t=-z \rbrace$, the process $(M_t)$ defined by:
$$M_t:=Q(Y_{t \wedge \tau_x})\exp\left(\beta \int_0^{\tau_x} \left(\frac{G\left(Q\left(Y_s\right)\right)}{Q(Y_s)}-1\right)\mbox{d}s\right)$$
replaces the process $(M_t)$ in \cite{harris2006further}. The rest of the proof of Theorem \ref{thequiv} is from now on different from \cite{harris2006further}. We begin by proving the case where $q=0$.
\begin{lemma}
\label{lemmaequiv}
Suppose $q=0$ and that $m=\infty$ or $\mu>-\mu_0$. Then there exists $C>0$ such that:
\begin{equation}\label{equivalent}
Q(x)=Ce^{-\lambda x}+\underset{x\rightarrow+\infty}{o}(e^{-\lambda x})
\end{equation}
and
\begin{equation}\label{equivalent2}
Q^{\prime}(x) \underset{x\rightarrow+\infty}{\sim} -C\lambda e^{-\lambda x},
\end{equation}
where $\lambda=\mu+\sqrt{\mu^2+2 \beta}>0$. 
\end{lemma} 
\begin{proof}
We have supposed that $p_0=p_1=0$, which implies that $G(s)\leq s^2, \forall s \in [0,1]$. Hence, for $0<\epsilon<\frac{\lambda}{2}$ and $y>0$, we have:
\begin{align}
 e^{\lambda y} G(Q(y)) &\leq e^{\lambda y} Q^2(y) \nonumber \\
 &\leq C_1e^{\lambda y} e^{-2(\lambda-\epsilon)y} \nonumber\\
 &\leq C_1 e^{-(\lambda-2\epsilon)y},\label{majelambQ}
\end{align}
where $C_1>0$. The second inequality in \eqref{majelambQ} is a consequence of Lemma \ref{majQinf}. Let $x\geq0$, we rewrite \eqref{KKPtempsarretx}:
\begin{align}
Q(x)e^{\lambda x}=1+\frac{\beta}{\alpha}&\left[\quad \; \; \int_0^x e^{\lambda y} G(Q(y))\mathrm{d}y \right. \label{Qint1}\\
&\quad +e^{2\alpha x}\int_x^{+\infty} e^{-(\alpha-\mu )y} G(Q(y))\mathrm{d}y \label{Qint2} \\
&\quad -\left.\int_0^{+\infty} e^{-(\alpha-\mu ) y} G(Q(y))\mathrm{d}y \qquad \;\right]. \label{Qint3}
\end{align}
The inequality $\eqref{majelambQ}$ implies that the integral terms in $\eqref{Qint2}$ and in $\eqref{Qint3}$ are well defined and that the integral term in $\eqref{Qint1}$ is convergent in $+\infty$. In the same way, the term in \eqref{Qint2} converges to $0$ when $x$ goes to infinity, and the term in \eqref{Qint3} does not depend on $x$. Hence, we have that:
\begin{equation}
\lim_{x \rightarrow + \infty} Q(x) e^{\lambda x}=1+\frac{\beta}{\alpha}\int_0^{+\infty}\left[e^{\lambda y}-e^{(\alpha-\mu)y}\right]G(Q(y))\mathrm{d}y,\label{temp1}
\end{equation}
which is $\eqref{equivalent}$. Now we will establish $\eqref{equivalent2}$ by differentiating $x \mapsto e^{\lambda x}Q(x)$:
\begin{align}
Q^{\prime}(x)e^{\lambda x}+\lambda Q(x)e^{\lambda x}&=\frac{\beta e^{\lambda x}G(Q(x))}{\alpha} \label{majQprime1}\\
&\quad +2\beta e^{2\alpha x} \int_{x}^{+\infty}e^{-(\alpha-\mu )y} G(Q(y))\mathrm{d}y\label{majQprime2}\\
&\quad -\frac{\beta e^{\lambda x}G(Q(x))}{\alpha}.\label{majQprime3}
\end{align}
The right-hand side of \eqref{majQprime1} cancel the term in \eqref{majQprime3}. Moreover, \eqref{majQprime2} can be bounded by using \eqref{majelambQ}. We then get that for $\epsilon>0$ small enough, there exists $C>0$ such that: 
\begin{equation}
2\beta e^{2\alpha x} \int_{x}^{+\infty}e^{-(\alpha-\mu )y} G(Q(y))\mathrm{d}y \leq Ce^{-(\lambda-2\epsilon)x}.
\end{equation}
Therefore,
\begin{equation}\label{temp2bis}
\lim_{x \rightarrow +\infty} 2\beta e^{2\alpha x} \int_{x}^{+\infty} \!  \!  \! e^{-(\alpha-\mu )y} G(Q(y))\mathrm{d}y =  \! \lim_{x \rightarrow +\infty}Q^{\prime}(x)e^{\lambda x}+\lambda Q(x)e^{\lambda x}=0
\end{equation}
and thus:
\begin{equation}\label{temp2}
\lim_{x \rightarrow +\infty}Q^{\prime}(x)e^{\lambda x}=-\lim_{x \rightarrow +\infty}\lambda Q(x)e^{\lambda x}.
\end{equation}
Equations \eqref{temp1} and \eqref{temp2} yield \eqref{equivalent2}. 
\end{proof}
Now, relying on Lemma \ref{lemmaequiv} we prove Theorem \ref{thequiv} dropping the hypothesis $q=0$.
\begin{proof}
 We consider for $x>0$ and $s \in [0,1]$: 
\begin{equation}
\tilde{Q}(x):=\frac{Q(x)-q}{1-q} \mbox{ and }\tilde{G}(s):=\frac{G((1-q)s+q)-q}{1-q}.
\end{equation}
It is easy to show that there exists $(\tilde{p}_i)\in ({\mathbb{R}_+})^{\mathbb{N}}$ such that $\tilde{p}_0=0$, $\sum_{i=0}^{\infty} \tilde{p}_i=1$ and $\tilde{G}(s)=\sum_{i=0}^{\infty} \tilde{p}_i s^i$. Besides, $\tilde{Q}(0)=1$, $\tilde{Q}(+\infty)=0$ and $\tilde{Q}$ solves the equation:
\begin{equation}
\frac{1}{2}y^{\prime \prime}(x)+\mu y^{\prime}(x)+\beta\left(\tilde{G}(y(x))-y(x)\right)=0, \; \forall{x >0},
\end{equation}
which means that $\tilde{Q}$ is the extinction probability of a branching Brownian motion with reproduction law $\tilde{L}$ with generating function $\tilde{G}$. The random variable $\tilde{L}$ has the following probabilistic interpretation, which can be found more precisely in \cite{athreya1972branching}. Consider a supercritical Galton-Watson tree with reproduction law $L$ and generating function $G$. If we condition the tree to survive and if we keep only the prolific individuals (that is these which give birth to an infinite tree) we obtain a Galton-Watson tree with reproduction law $\tilde{L}$. In the same way, the branching Brownian motion with with reproduction law $\tilde{L}$ without killing on a barrier is the branching Brownian motion with reproduction law $L$ without killing on a barrier conditioned to survive where we keep only the prolific individuals.

The assumptions of Lemma \ref{lemmaequiv} are almost satisfied. Furthermore, we can simply ignore reproduction events corresponding to $\tilde{p}_1$ and replace our branching Brownian motion with rate $\beta$ and reproduction law described by $\tilde{G}$ by one where the branching rate is $(1-\tilde{p}_1)\beta$ and the reproduction law is described by the generating function:
\begin{equation}
G_2(s):=\frac{\tilde{G}(s)-\tilde{p}_1 s}{1-\tilde{p}_1}.
 \end{equation}
 We now have that:  $G_2(0)=G_2^{\prime}(0)=0$ and we thus can apply Lemma $\ref{lemmaequiv}$:
 \begin{equation}
 \frac{Q(x)-q}{1-q}=\tilde{Q}(x)\underset{x \rightarrow +\infty}{\sim}ke^{-\tilde{\lambda} x},
 \end{equation}
 where $k>0$ and $\tilde{\lambda}=\sqrt{2\beta(1-\tilde{p}_1)+\mu^2}+\mu=\sqrt{2\beta(1-G^{\prime}(q))+\mu^2}+\mu$. 
\end{proof}
 
We now consider $J$, the maximal open subinterval  of $I$ (defined in Proposition \ref{extension}) such that $(0,\infty) \subset J$  and on which $Q$ is decreasing. 
\begin{proposition}\label{premdist}
Fix $\mu \in (-\mu_0, +\infty)$. Let $J$ be defined as above. We have $x_0(\mu):=\inf J >-\infty$ and either $Q^{\prime}(x_0(\mu),\mu)=0$ or  $Q(x_0(\mu),\mu)=R_G$.
\end{proposition}
\begin{proof}
We begin by proving that $x_0(\mu)$ is finite. First, suppose that $R_G=1$. By definition of $I$, we have in this case $I=(0,+\infty)$ and thus $J=I$ which implies $x_0(\mu):=\inf J=0$. Furthermore, $Q(x_0(\mu),\mu)=R_G=1$. Suppose, now that $R_G>1$. By Proposition \ref{extension}, this implies that $I$ is strictly bigger than (0,$+\infty$). Moreover, $Q$ is decreasing on $(0,+\infty)$ and $Q^{\prime}(0)\neq 0$ because $Q^{\prime}(0)= 0$ would imply that $Q\equiv 1$ by Cauchy-Lipschitz Theorem. Therefore $Q^{\prime}(0)<0$. This implies that $Q$ is decreasing on a interval of the form $[y,0)$, with $y<0$. Suppose that the lower bound of $J$ is $-\infty$ (or equivalently that $J=I=\mathbb{R}$). Since on $I$, $Q(x)<R_G$, and since $Q$ is decreasing on $\mathbb{R}$, there exists $l\in (1,R_G]$ such that:
\begin{equation}\label{limitabsurde}
\lim_{x \rightarrow - \infty} Q(x)=l.
\end{equation}

Let $\epsilon \in (0,l-1)$. By the Intermediate Value Theorem, there exists $x_\epsilon \in \mathbb{R}^-$ such that $Q(x_\epsilon)=l-\epsilon>1$.  
Moreover, by integrating $\eqref{KKPx}$, we get that there exists $C \in \mathbb{R}$, such that:
\begin{equation}\label{limitabsurde15}
\frac{1}{2}Q^{\prime}(x)+\mu Q(x)+\int_{x_{\epsilon}}^{x}\beta\left(G(Q(x))-Q(x)\right)\mathrm{d}x=C, \quad \forall x \in \mathbb{R}.
\end{equation}
Since for $x<x_\epsilon$,  $1<Q(x)\leq R_G$ and since $u \mapsto G(u)-u$ is increasing and positive on $(1,R_G)$, Equation \eqref{limitabsurde15}  yields:
\begin{align}
Q^{\prime}(x)\geq & -2\int_{x_\epsilon}^x\beta\left(G(Q(x))-Q(x)\right) \mathrm{d}x+2C-2|\mu|R_G \nonumber \\
\geq &-2\left(x-x_\epsilon\right)\beta\left(G(l-\epsilon)-l-\epsilon\right) +2C-2|\mu|R_G.
\label{limitabsurde2}
\end{align}
Consequently, $\lim_{x \rightarrow -\infty} Q^{\prime}(x)=+\infty$ which is in contradiction with the fact that $Q$ is decreasing on $\mathbb{R}$. Therefore, $x_0(\mu)$ is finite. \\

Let us prove the last part of the proposition. Suppose that $Q(x_0(\mu),\mu) \neq R_G$. In this case $G(Q(y))<+\infty, \; \forall y \in [x_0(\mu),0]$. Hence, \eqref{Qprimeequ} yields that:
$$\limsup_{x \downarrow x_0(\mu)} \left| Q^{\prime}(x)  \right| =\lim_{x \downarrow x_0(\mu)} \left| Q^{\prime}(x) \right|<+\infty.$$
Thus, by Proposition \ref{extension}, $x_0(\mu)\neq x_l$. Therefore, by definition of $x_0(\mu)$,  there exists $\epsilon>0$ such that $Q$ is defined and smooth on an interval $J_{\epsilon}=(x_0(\mu)-\epsilon, x_0(\mu)+\epsilon)$, not decreasing on $(x_0(\mu)-\epsilon,x_0(\mu))$ and decreasing on $(x_0(\mu),  x_0(\mu)+\epsilon)$. We then have that $Q^{\prime}(x_0(\mu),\mu)=0$. Hence, either $Q^{\prime}(x_0(\mu),\mu)=0$ or  $Q(x_0(\mu),\mu)=R_G$.
\end{proof}
As we said before, we will often write $x_0$ instead of $x_0(\mu)$ in what follows. By definition of $J$, $Q$ is decreasing on $J$. This implies that $Q^{\prime}(x)<0$ for all $x\in J\setminus S$, where $S$ is a subset of $J$ without accumulation points. We show that $S$ is actually empty.
\begin{proposition}\label{propqprimeinfstrict0}
For all $x\in J$, $Q^{\prime}(x)<0$.
\end{proposition}
\begin{proof}
We recall that $J$ is an open interval and thus $x_0$ is not included in $J$. Consider now $x_1\in J$ and suppose that $Q^{\prime}(x_1)=0$. The KPP equation \eqref{KKPx} yields that:
\begin{equation}\label{QsecondabsurdeQprimeegal0}
Q^{\prime \prime}(x_1)=-2\beta\left(G(Q(x_1))-Q(x_1)\right).
\end{equation}

If $x_1\in (x_0,0)$, then $Q^{\prime \prime}(x_1)<0$ since $Q(x) > Q(0)=1, \forall x\in (x_0,0)$ and since $G(s)>s, \forall s>1$. Therefore $x_1$ is a local maximum, which is in contradiction with the fact that on $J$, $Q$ is decreasing. 

Similarly, if $x_1\in (0,+\infty)$, then $Q^{\prime \prime}(x_1)>0$ which also contradicts the decrease of $Q$ on $J$. 

We have already proved that $Q^{\prime}(0)\neq 0$. Therefore for all $x\in J$, $Q^{\prime}(x)<0$. 
\end{proof}
\section{Radius of convergence}
In this section, we will focus on $R(\mu)$ the radius of convergence of $f_x$. We will show that it is a function of $\mu$ which does not depend of $x>0$ and determine how this radius evolves with respect to $\mu$. As a first step, we bound $R(\mu)$. 
\begin{proposition}\label{firstboundradius}
For any $\mu \in \mathbb{R}$, we have:
$$1\leq R(\mu)\leq R_G.$$
\end{proposition}
Observe that the case where $\mu \leq -\mu_0$ is trivial. Indeed, $R_G$ is always greater or equal to $1$ (since $G$ is a generating function) and $R(\mu)=1$ for this range of $\mu$ (see \cite{maillard2013number}).
\begin{proof}
The fact that $R(\mu)\geq 1$ is obvious since $f_{x}(1)=Q(x)<+\infty,\forall x\geq 0$.
Let $D_x$ be defined as the total number of birth event (which include the case $L=0$) before $\zeta_x$ and fix $k \in \mathbb{N}\setminus \lbrace 1 \rbrace$. To prove that $R(\mu)\leq R_G$, we will calculate $\tilde{q}_k(x)=\mathbb{P}\left(D_x=1,Z_x=k,\zeta_x<+\infty\right)$. Since these computations are very similar to those of Lemma  \ref{thKPPtempsarretx}, we will skip some details. Like in Lemma \ref{thKPPtempsarretx}, we denote by $(K^x_s)$ a Brownian motion with drift $\mu$ starting from $x$ and killed at $0$ and by $T_1$ an exponential random variable with parameter $\beta$. We also recall that $p_k=\mathbb{P}\left(L=k\right)$, $\alpha=\sqrt{2\beta+\mu^2}$ and $\lambda=\alpha+\mu$. Using Equations \eqref{divapresreach} and \eqref{dKX}, we get:
\begin{align}
\mathbb{P}\left(D_x=1,Z_x=k,\zeta_x<+\infty\right)&= \int_0^{+\infty}\mathbb{P}\left(K^x_{T_1}\in \mathrm{d}y\right)p_k\mathbb{P}\left(T_1\geq\zeta_{y}\right)^k \nonumber \\
&=\! \frac{\beta p_k}{\alpha} \! \int_0^{+\infty} \! \! \! \!  \! \! e^{-\mu x}e^{\mu y}\left[e^{-\alpha |y-x|}-e^{-\alpha (y+x)}\right] \! e^{-\lambda k y}\mathrm{d}y \nonumber \\
&= 2 \beta p_k \left[\frac{e^{-\lambda x}-e^{-k\lambda x}}{\lambda(k-1)\left(\alpha-\mu+\lambda k\right)}\right]. \label{dxzxzetax}
\end{align}
We know that multiplying the coefficients of a power series by a rational function does not change its radius of convergence. Furthermore, the term $e^{-\lambda x}-e^{-k\lambda x}$ is equivalent to $e^{-\lambda x}$ when $k$ goes to infinity and $x>0$. Therefore, the radius of convergence of the power series whose coefficients are the left-hand side of \eqref{dxzxzetax}, is $R_G$. Since
\begin{equation}
\mathbb{P}\left(D_x=1,Z_x=k,\zeta_x<+\infty\right)\leq q_k(x)=\mathbb{P}\left(Z_x=k,\zeta_x<+\infty\right),
\end{equation}
we can easily show, for instance with Cauchy-Hadamard Theorem \eqref{HCth}, that $R(\mu)\leq R_G$. 
\end{proof}
This proposition proves in particular that $R_G=1$ implies that $R(\mu)=1$ for any $\mu \in \mathbb{R}$. That is why we can suppose that $R_G>1$ (which implies that $m< +\infty$) throughout this section. We now prove Theorem \ref{thcaractRwrtmu}.
\begin{proof}
Fix $s_0=Q(x_0)$. By Proposition \ref{premdist}, $Q^{\prime}(x)<0, \; \forall x \in (x_0,+\infty)$ and $Q$ is continuous on $(x_0,+\infty)$ and right-continuous at $x_0$. Therefore, $Q^{-1}$ is well-defined on $(q,s_0)$. By right-continuity of $Q$ we even have that $Q^{-1}(s_0)=x_0$. Furthermore, we recall that for $x>0$:
\begin{equation}
f_x(s)=Q(Q^{-1}(s)+x)\mbox{, } \forall s \leq R(\mu)\wedge s_0. \tag{\ref{linkomegQ}}
\end{equation} 
Let us show that $s_0\leq R(\mu)$. Suppose not. Let us define $x_1=Q^{-1}(R(\mu))$. By Proposition $\ref{extension}$ there is a complex neighbourhood $V_1$ of $x_1$ such that $Q$ admit an analytic continuation on $V_1$. Furthermore, $Q^{\prime}(x_1)\neq 0$, which yields, by Theorem 10.30 of \cite{rudin1987real}, the existence of $V_2$, a complex neighbourhood of $x_1$ included in $V_1$ such that $Q$ admit a complex analytic inverse on $V_2$. Let us call $\phi$ the analytic continuation of $Q^{-1}$ on $Q(V_2)$. We  now fix $x >0$. Similarly, there exists $V_3$ a neighbourhood of $x_1$ such that $Q$ admit an analytic continuation on the open $V_3+x$. Vivanti-Pringsheim's Theorem \cite{hille1962analytic} ensures that if $R(\mu)$ is the radius of convergence of $f_x(s)$ then $s \mapsto f_x(s)$ cannot have an analytic extension around $R(\mu)$. But $s \mapsto Q(\phi(s)+x)$ is precisely such an extension on $Q(V_2) \cap Q(V_3)$. This is in contradiction with the assumption that $R(\mu)<s_0$ and therefore we have: 
\begin{equation}\label{encadrements0s1RMU} 
s_0 \leq R(\mu).
\end{equation}
We will now prove that $s_0=R(\mu)$. By Proposition \ref{premdist}, we have $Q(x_0)=R_G$ or $Q^{\prime}(x_0)=0$.

Suppose that $s_0=Q(x_0)=R_G$, we have by \eqref{encadrements0s1RMU} that $s_0=R_G\leq R(\mu)$. Proposition \ref{firstboundradius} tell us that $R(\mu)\leq R_G$, and thus $R(\mu)= R_G=Q(x_0)$ in this case. 

Suppose now that $Q^{\prime}(x_0)=0$. On $(q,s_0)$ we have that:
\begin{equation}\label{derivresps}
f^{\prime}_x(s)={(Q^{-1})}^{\prime}(s)Q^{\prime}(Q^{-1}(s)+x),
\end{equation}
where $f^{\prime}_x$ is the derivative of $f_x$ with respect to $s$. Let us prove by contradiction that $s_0\geq R(\mu)$. Suppose that $s_0<R(\mu)$. Since the radius of convergence of $f_x$ is strictly greater than $s_0$, the left-limit in $s_0$ of the left-hand side of $\eqref{derivresps}$ tends to a finite limit. However,  since we suppose that $Q^{\prime}(x_0)=0$ we have that 
$$\lim_{s \uparrow s_0} (Q^{-1})^{\prime}(s)=-\infty$$
and since $\left|Q^{\prime}(Q^{-1}(s_0)+x) \right|<0$, the left-limit of the right-hand side of $\eqref{derivresps}$ is not finite, which is a contradiction. Therefore, $s_0\geq R(\mu)$. This fact and \eqref{encadrements0s1RMU} yield: 
$$
Q(x_0) = R(\mu).
$$
\end{proof}
We thus have proved that the radius of convergence is $Q(x_0,\mu)$. We want now to focus on the variation of $R(\mu)$ with respect to $\mu$. For this purpose, inspired by Maillard's approach \cite{maillard2013number}, we introduce a new object $a$, defined by: 
\begin{equation}\label{defa} 
a(s,\mu):=\partial_x f_0(s)=Q^{\prime}(Q^{-1}(s)), \forall s \in \left(q,Q\left(x_0\left(\mu\right)\right)\right),
\end{equation}
where the second equality is given by \eqref{linkomegQ} and justifies the existence of $a$.
Although in this article we will only see $a$ as a mean to simplify some proofs, there are deeper reasons for its use. We know, see for instance Neveu \cite{neveu1988multiplicative}, that $(Z_x)_{x\geq 0}$ is a Galton-Watson process. Let us consider its infinitesimal generator defined by $b(s):=\partial_x F_0(s)$ (this is $a$ in  \cite{maillard2013number}), where $F$ is defined as in \eqref{defbigF}. As we will use $f$ instead of $F$, $a$ will be a slightly different object, which will be nevertheless identical to $b$ when $\mu<\mu_0$ and $s \in (q,1)$, and which will satisfy the same properties. The function $a$ is also a power series with radius of convergence $R(\mu)$. Since we do not use this fact in the present work, we will not prove it. Besides, we can notice that the definition of $a$ implies that the trajectory of $X$ (defined in $\eqref{defX}$) for $x\in(x_0,+\infty)$ is the same of the one of $s\mapsto (s,a(s))$ for $s \in \left(q,Q\left(x_0\right)\right)$. We can therefore work with either of them, depending on the situation. Finally, the travelling wave equation \eqref{KKPx} and the definition of $a$ yield:
\begin{equation}\label{eqa}
a^{\prime}(s)a(s)=-2\mu a(s)-2\beta\left(G(s)-s\right) \forall s\in (q,Q(x_0)),
\end{equation}
where $x_0$ is defined in Proposition \ref{premdist}. Furthermore, the definition of $a$ $\eqref{defa}$ and Proposition \ref{propqprimeinfstrict0} implies that:
\begin{equation} \label{aneq0}
a(s)< 0, \forall s \in  (q,Q(x_0)).
\end{equation}
Therefore, the application of the results in Section 12.1 of \cite{ince1927ordinary} to \eqref{eqa} yields that for every $s\in  (q,Q(x_0))$, we can analytically extend $a$ to a complex neighbourhood of $s$. \\

We can now use $a$ to determine the variation of $R$ with respect to $\mu$.
\begin{proposition} \label{nocut}
 Fix $\mu_1$ and $\mu_2$ such that $-\mu_0<\mu_1<\mu_2$. Set $r=R(\mu_1) \wedge R(\mu_2)$ and define $\phi(s):=a(s,\mu_1)-a(s,\mu_2)$ for $s \in[q, r)$. The function $\phi$ is positive on $(q,r)$ and increasing on $(1,r)$. Therefore, $R$ is a non-decreasing function on $\mathbb{R}$ and more specifically an increasing function at each $\mu$ such that $R(\mu) \in (1,R_G)$.
\end{proposition}
\begin{proof}
We will prove by contradiction that $\phi$ is positive. Let us define 
\begin{equation}
H:=\lbrace s \in (q,r), a(s,\mu_1)\leq a(s,\mu_2) \rbrace
\end{equation} 
and suppose that $H$ is non-empty. We can then define $h:= \inf H<+\infty$. The functions $a(\cdot ,\mu_1)$ and $a(\cdot ,\mu_2)$ are continuous on $(q,r)$ and 
$$\lim_{s \rightarrow q^+} a(s ,\mu_1)=\lim_{s \rightarrow q^+} a(s ,\mu_2)=0,$$
by definition of $a$ \eqref{defa}. Hence, $h \in H \cup \lbrace q \rbrace$. We now need an equivalent of $a$ when $s$ goes to $q$. For $s>q$, fix $x=Q^{-1}(s)$. We recall that for $\mu \in \mathbb{R}$, we define $\tilde{\lambda}$ by $\tilde{\lambda}=\sqrt{2\beta(1-G^{\prime}(q))+\mu^2}+\mu$. With the help of \eqref{equivalent}, we get:
\begin{align*}
Q(x)&=s \\
q+Ce^{-\tilde{\lambda}x}+\underset{x\rightarrow+\infty}{o}\left(e^{-\tilde{\lambda}x}\right)&=s\\
-\tilde{\lambda}x+\log(C)+\log\left(1+\underset{x\rightarrow+\infty}{o}\left(e^{-\tilde{\lambda}x}\right)\right)&=\log(s-q) \\
Q^{-1}(s)=x&=-\frac{\log(s-q)}{\tilde{\lambda}}+\frac{\log(C)}{\tilde{\lambda}}+\underset{x\rightarrow+\infty}{o}(1).
\end{align*}
Similarly, using  \eqref{equivalent2}, we finally obtain:
\begin{equation}
a(s)=Q^{\prime}(Q^{-1}(s))\sim -\tilde{\lambda}(s-q) \mbox{ as } s\downarrow q.
\end{equation}
Furthermore, $\tilde{\lambda}$ is increasing with respect to $\mu$ and thus there exists a neighbourhood $V$ of $q$ such that:
\begin{equation}\label{suppresq}
a(s,\mu_1)> a(s,\mu_2), \forall s\in V\cap (q,1),
\end{equation} 
and thus $h>q$. This fact and the fact that $a(\cdot,\mu_1)$ and $a(\cdot,\mu_2)$ are continuous imply in particular:
\begin{equation}\label{ahmu1mu2}
a(h,\mu_1)= a(h,\mu_2).
\end{equation} 
By recalling \eqref{aneq0}, we know that $a(s,\mu_1)\neq 0$ and $a(s,\mu_2)\neq 0$, $\forall s \in (q,r)$. Equation $\eqref{eqa}$  thus implies that for $s \in (q,r)$:
\begin{equation}\label{difftwoa}
a^{\prime}(s,\mu_1)-a^{\prime}(s,\mu_2)=2(\mu_2-\mu_1)-2\beta\left(\frac{1}{a(s,\mu_1)}-\frac{1}{a(s,\mu_2)}\right)\left(G(s)-s\right).
\end{equation}
In particular taking $s=h$ in $\eqref{difftwoa}$, we have by using \eqref{ahmu1mu2}:
\begin{equation}\label{difftwoh}
a^{\prime}(h,\mu_1)-a^{\prime}(h,\mu_2)=2(\mu_2-\mu_1)>0.
\end{equation}
Equations \eqref{ahmu1mu2} and \eqref{difftwoh} and the fact that $h>q$ yield that there exists $\epsilon>0$ such that $a(s,\mu_1)\leq a(s,\mu_2), \; \forall s\in (h-\epsilon,h)$, which contradicts the definitions of $H$ and $h$. Hence $H$ is empty and: 
\begin{equation} \label{a1greata2}
a(s,\mu_1)>a(s,\mu_2), \; \forall s \in (q,r).
\end{equation} 
We have thus proved that $\phi$ is positive on $(q,r)$. Furthermore, since $G(s)> s, \forall s \in (1,R_G)$, Equation \eqref{difftwoa} yields that $\phi^{\prime}>0$ on $(1,r)$. \\

Let us now focus on $\mu \mapsto R(\mu)$. If $R(\mu_2)=R_G$, we know by Proposition \ref{firstboundradius} that $R(\mu_1)\leq R(\mu_2)$. 

Now, suppose that $R(\mu_2)<R_G$. Proposition \ref{premdist} yields $Q^{\prime}(x_0(\mu_2),\mu_2)=a(R(\mu_2),\mu_2)=0$.  Let us first prove by contradiction that $R(\mu_2)\geq R(\mu_1)$.  If $R(\mu_1)> R(\mu_2)$, then $a(R(\mu_2),\mu_1)$ is well-defined. Furthermore, by taking the left-limit when $s$ goes to $r=R(\mu_2)$ in Equation \eqref{a1greata2}, we get that 
\begin{equation}\label{armu2mu1geqmu2mu2}
a(R(\mu_2),\mu_1)\geq a(R(\mu_2),\mu_2)=0.
\end{equation}
 Since $R(\mu_1)> R(\mu_2)$, Equation \eqref{armu2mu1geqmu2mu2} is in contradiction with \eqref{aneq0}. Hence,  $R(\mu_1)\leq R(\mu_2)<R_G$. 
 
 Let us now prove that  $R(\mu_1)< R(\mu_2)$. Since $R(\mu_1)<R_G$, Proposition \ref{premdist} yields $a(R(\mu_1),\mu_1)=0$. Moreover, the increase of $\phi$ implies: 
\begin{equation}
-a(R(\mu_1),\mu_2)=a(R(\mu_1),\mu_1)-a(R(\mu_1),\mu_2)>a(1,\mu_1)-a(1,\mu_2)>0.
\end{equation}
Consequently,  $R(\mu_1)$ cannot be equal to $R(\mu_2)$ and thus $R(\mu_1)< R(\mu_2)$.
\end{proof}
We will now prove the continuity of $R$. As a first step, we will prove the continuity of $X(x,\cdot )=(Q(x,\cdot ),Q^{\prime}(x,\cdot ))\in\mathbb{R}^2$ for fixed $x \in \mathbb{R}^+$ by probabilistic methods. In fact, for our purpose, it would be sufficient to show the continuity of $\mu \mapsto Q^{\prime}(0,\mu)$. But if we have the continuity of $Q$, it is simple to establish the continuity of $Q^{\prime}$ with respect to $\mu$. Next, using the fact that $Q$ is solution of KPP, we will extend the continuity of $Q$ to negative half-line and deduce from it the continuity of $R$.
\begin{proposition} \label{lemcontinuQprilme0}
For any $x \geq 0$, the functions $\mu \mapsto Q(x,\mu)$ and  $\mu \mapsto Q^{\prime}(x,\mu)$ (where the derivative is with respect to $x$) are continuous on $\mathbb{R}$.
\end{proposition}
We recall we suppose throughout this section that $R_G>1$ (which implies in particular that $m<\infty$ and that the process cannot explode in finite time). Although Proposition \ref{lemcontinuQprilme0} holds in general, this assumption allow us to avoid unnecessary technical complications.
\begin{proof}
Fix $x_1>0$. To prove the continuity of $Q$ and $Q^{\prime}$ with respect to $\mu$ it is easier to consider a branching Brownian motion starting from 0 without drift killed on the barrier:
\begin{equation}
\gamma_{x_1,\mu}=\lbrace(t,x)\in \mathbb{R}^2, x=-\mu t-x_1\rbrace.
\end{equation} 
rather than a branching Brownian motion with drift $\mu$ and killed at $-x_1$. We can thus move the barrier by changing $\mu$ for a fixed $\omega \in \mathcal{T}$. \\

Let us fix some notations. We call $\mathcal{N}_t$ the set of particles alive at time $t$ without killing and for $\mu \in \mathbb{R}$, we call $\mathcal{S}^{\mu}_t$ the set of particles stopped on $\gamma_{x_1,\mu}$ at time $t$,  $\mathcal{A}^{\mu}_t$ the set of particles alive for the branching Brownian motion with killing on $\gamma_{x_1,\mu}$ and $Z_{x_1,\mu}:=|\mathcal{S}^{\mu}_t|$, that is the number of particles killed on $\gamma_{x_1,\mu}$. For $u\in \mathcal{N}_t$ and $s\leq t$, we call $X_u(s)$ the position of the ancestor of $u$ alive at time $s$. We denote by $K_{\mu}$ the event $\lbrace \zeta_{x_1,\mu}<+\infty\rbrace$ which means "All particles are killed on $\gamma_{x_1,\mu}$". We thus have $Q(x_1,\mu)=\mathbb{P}(K_{\mu})$. The function $\mu \mapsto Q(x_1,\mu)$ is non-increasing because if $\mu_1 \leq \mu_2$ then $K_{\mu_2} \subset K_{\mu_1}$.
Therefore, $\mu \mapsto Q(x_1,\mu)$ has a left-limit and right-limit at every point. \\

We temporarily suppose  that $p_0=G(0)=0$. To prove the continuity of $\mu \mapsto Q(x_1,\mu)$ let us start by proving its left-continuity.  The left-continuity for $\mu\leq -\mu_0$ is obvious, since for this range of $\mu$, $Q(x,\mu)=1$, $\forall x \in \mathbb{R}^+$.
Suppose that  $\mu \mapsto Q(x_1,\mu)$ is not left-continuous for a $\mu_1 >-\mu_0$, which is equivalent to the fact that:
\begin{equation}
\mathcal{L}_{\mu_1}:=\bigcap_{\mu<\mu_1}K_{\mu}\cap K^{c}_{\mu_1}=\bigcap_{\mu<\mu_1, \mu \in \mathbb{Q}}K_{\mu}\cap K^{c}_{\mu_1}.
\end{equation} 
happens with non-zero probability (the second equality ensures that $\mathcal{L}_{\mu_1}$ is measurable). Fix $\omega \in \mathcal{L}_{\mu_1}$. Since we have supposed that $G(0)=0$, the function $\mu \mapsto Z_{x_1,\mu}(\omega)$ is non-decreasing on $(-\infty,\mu_1)$. It thus has  a right-limit when $\mu$ goes to $\mu_1$, $l(\omega)\in \mathbb{R} \cup \lbrace +\infty \rbrace$. We will first prove that this limit is infinite. Fix $M\in \mathbb{N}$. On $K^{c}_{\mu_1}$ (and consequently on $\mathcal{L}_{\mu_1}$) we know that almost surely the number of particles in  $\mathcal{A}^{\mu_1}_t$ increases to infinity (see for instance \cite{kesten1978branching}) as $t$ tends to infinity. Therefore, for almost every $\omega \in \mathcal{L}_{\mu_1}$ there exists $t(\omega)>0$ such that $N_{x_1,\mu_1}(t(\omega))$, the number of particles alive at time $t(\omega)$, is larger than $M$. Now, fix:
\begin{equation}
\epsilon:=\inf_{u \in \mathcal{A}^{\mu_1}_t} \lbrace X_u(s)+\mu_1 s +x_1, s\leq t(\omega)\rbrace \mbox{ and } \mu_2(\omega):=\mu_1-\frac{\epsilon}{2t(\omega)}.
\end{equation}
$\epsilon$ is the infimum of a finite number of strictly positive continuous functions and thus is strictly positive.  $\mu_2(\omega)$ has been chosen such that $\mu_2(\omega)<\mu_1$ and such that every particle which have not been killed on $\gamma_{x_1,\mu_1}$ before $t(\omega)$ is not killed on $\gamma_{x_1,\mu_2(\omega)}$ before $t(\omega)$ either. Since $G(0)=0$ and $\omega \in \mathcal{L}_{\mu_1} \subset K_{\mu_2(\omega)}$, we will necessarily have $Z_{x_1,\mu_2(\omega)}(\omega) \geq N_{x_1,\mu_1}(t(\omega))\geq M$ and thus $l(\omega) \geq M$. As $M$ is arbitrary, we have for almost every $\omega \in \mathcal{L}_{\mu_1}$: 
\begin{equation}\label{limZmuleftQ}
\lim_{\mu \rightarrow \mu_1^-}Z_{x_1,\mu}(\omega)=+\infty.
\end{equation}
The fact that $\mathbb{P}(\mathcal{L}_{\mu_1})>0$ and \eqref{limZmuleftQ} imply that:
\begin{equation}\label{L1Zinf}
\mathbb{E}\left[\liminf_{\mu \rightarrow \mu_1^-} \mathbf{1}_{\mathcal{L}_{\mu_1}}Z_{x_1,\mu}\right]=+\infty.
\end{equation}
 Furthermore, by Fatou's lemma we get:
\begin{align}
\mathbb{E}\left[\liminf_{\mu \rightarrow \mu_1^-} \mathbf{1}_{\mathcal{L}_{\mu_1}}Z_{x_1,\mu}\right]&\leq \liminf_{\mu \rightarrow \mu_1^-} \mathbb{E}\left[\mathbf{1}_{\mathcal{L}_{\mu_1}}Z_{x_1,\mu}\right] \label{firstineqleftcontQ}\\
&\leq \liminf_{\mu \rightarrow \mu_1^-} \mathbb{E}\left[\mathbf{1}_{\lbrace \zeta_{x_1,\mu}<+\infty \rbrace}Z_{x_1,\mu}\right] \label{subsetineqleftcontQ}\\
&\leq \liminf_{\mu \rightarrow \mu_1^-} \frac{Q^{\prime}(x_1,\mu)}{Q^{\prime}(0,\mu)}. \label{fracleftcontQ}
\end{align}
Inequality \eqref{subsetineqleftcontQ} comes from the definition of $\mathcal{L}_{\mu_1}$ which implies that for any $\mu<\mu_1$, $\mathcal{L}_{\mu_1} \subset \lbrace \zeta_{x,\mu}<+\infty \rbrace$ and \eqref{fracleftcontQ} is obtained by the differentiation of $f_{x_1}$ with respect to $s$ at $1$ and from \eqref{linkomegQ}.
We recall that for $x\geq 0$, $q \leq Q(x)\leq 1$ and for $s \in (q,1)$, $G(s)\leq s$. Therefore, for any $\mu>-\mu_0$, \eqref{Qprimeequ} yields:
\begin{equation}
\frac{Q^{\prime}(x_1,\mu)}{Q^{\prime}(0,\mu)}\leq e^{-2\mu x_1}.
\end{equation}
The left-hand of \eqref{firstineqleftcontQ} is thus bounded by $e^{-2\mu_1 x_1}$, which contradicts  \eqref{L1Zinf}. Hence, $\mathbb{P}(\mathcal{L}_{\mu_1})=0$ and consequently $\mu \mapsto Q(x_1,\mu)$ is left-continuous.\\

Suppose now that $\mu \mapsto Q(x_1,\mu)$ is not right-continuous at $\mu_1 \geq -\mu_0$. That implies that the event $\mathcal{R}_{\mu_1}$, defined by:
\begin{equation}
\mathcal{R}_{\mu_1}:=\bigcap_{\mu>\mu_1}K^{c}_{\mu}\cap K_{\mu_1}=\bigcap_{\mu>\mu_1,\mu\in \mathbb{Q}}K^{c}_{\mu}\cap K_{\mu_1}
\end{equation} 
happens with positive probability.  We define for $u\in \mathcal{N}_t$: $Y_u(t)=X_u(t)+\mu_1 t+x_1$ and $ \tau_u=\inf \lbrace s\leq t, Y_u(s)=0 \rbrace$. Furthermore, we call $H$ the event:

\begin{equation}
H:=\underset{n\in \mathbb{N}^*}{\bigcap}\underset{\; \; u\in \mathcal{S}^{\mu_1}_n}{\bigcap}  \underset{\; s \in (0,n-\tau_u)\cap \mathbb{Q}  }{\bigcup}  \lbrace Y_u(\tau_u+s)<0\rbrace.
\end{equation}
 Let us briefly show that $\mathbb{P}(H^c)=0$. Let $B$ be a Brownian motion and $\tau=\inf \lbrace s \in \mathbb{R}^+, B_s=-\mu_1 s-x_1 \rbrace$. A Brownian motion cannot stay above a barrier after reaching this barrier and the many-to-one lemma (see for instance Theorem 8.5 of \cite{hardy2006new}) will ensure that none particle of the branching Brownian motion can do it. More formally, we have:
\begin{align}
\mathbb{P}(H^c)&=\mathbb{P}\left(\underset{n\in \mathbb{N}^*}{\bigcup}\underset{\; \; u\in \mathcal{S}^{\mu_1}_n}{\bigcup}  \underset{\; s \in (0,n-\tau_u)\cap \mathbb{Q} }{\bigcap}  \lbrace Y_u(\tau_u+s)\geq 0\rbrace \right) \nonumber \\
&\leq \sum_{n=1}^{+\infty}\mathbb{E}\left(\sum_ {u\in \mathcal{S}^{\mu_1}_n}  \mathbf{1}_{\lbrace Y_u(\tau_u+s )\geq 0, \; \forall s  \in (0,n-\tau_u)\cap \mathbb{Q} \rbrace} \right) \nonumber\\
&\leq \sum_{n=1}^{+\infty}  e^{-\beta (m-1) n} \mathbb{P}\left(\tau<n ; Y_{\tau+s}\geq 0, \; \forall s  \in (0,n-\tau)\cap \mathbb{Q} \right) \label{Hcineg3} \\
&\leq 0 \label{Hcineg4},
\end{align} 
where $Y$ is a Brownian motion with drift $\mu_1$ starting from $x_1$. Inequality \eqref{Hcineg3} is just many-to-on lemma and inequality  \eqref{Hcineg4} comes from the strong Markov property.

We now fix $n=\lfloor \zeta_{x_1,\mu_1} \rfloor +1$. Since $n\geq \zeta_{x_1,\mu_1}$ and $\mathbb{P}(H)=1$, we have on $\mathcal{R}_{\mu_1}$ that for all $u \in \mathcal{N}_n$ there exists $s_u \in (0,n)$ such that $Y_u(s_u)<0$. The fact that $\mathcal{N}_n$ is finite implies there exists $\epsilon>0$ such that $Y_u(s_u)\leq -\epsilon, \forall u \in \mathcal{N}_n$. If we take $\mu=\mu_1+\frac{x_1}{n}+\frac{\epsilon}{2n}$ then each particle of $\mathcal{N}_n$ reaches $\gamma_{x_1,\mu}$ before $n$, which means that the process dies on  $\gamma_{x_1,\mu}$. This is in contradiction with the definition of $\mathcal{R}_{\mu_1}$. Therefore, 
\begin{equation}
\mathbb{P}\left(\mathcal{R}_{\mu_1}\right)=0.
\end{equation}
We have proved that $\mu \mapsto Q(x_1,\mu)$ is also right-continuous and thus continuous in the case where $G(0)=0$. The argument of the proof of Theorem $\ref{thequiv}$, which consists of looking the tree of prolific individuals can again be applied to prove the result in the general case. \\

Let us now prove the continuity of $\mu \mapsto Q^{\prime}(x,\mu)$ for any $x>0$. We can deduce from \eqref{KKPtempsarretx} that:
\begin{equation}\label{Qprime0}
Q^{\prime}\left(0,\mu\right)=2 \alpha\int_0^{+\infty}e^{-(\alpha-\mu)y}G(Q(y,\mu))\mathrm{d}y-\lambda,
\end{equation}
where we recall that $\alpha=\sqrt{\mu^2+2\beta}$ and $\lambda=\mu+\alpha$. Let $K$ be a compact subset of $\mathbb{R}$.  For any $\mu \in K$ and for any $y \geq 0$, we have:
\begin{equation}
e^{-(\alpha-\mu)y}G(Q(y,\mu))\leq e^{-Cy},
\end{equation} 
where $\displaystyle C=\min_{\mu \in K}\lbrace \alpha-\mu \rbrace>0$. Moreover, we know that  $\mu \mapsto Q(x,\mu)$ is continuous and, as a composition of continuous functions, $\mu \mapsto e^{-(\alpha-\mu)y}G(Q(y,\mu))$ is also continuous. Therefore, $\mu \mapsto Q^{\prime}\left(0,\mu\right)$ is continuous. Similarly, with the help of \eqref{Qprimeequ}, we can easily prove  that $\mu \mapsto Q^{\prime}(x,\mu)$ is continuous.

\end{proof}
The continuity of $Q$ with respect to $\mu$ will be useful to prove the continuity of $R$. Before proving this continuity, we just prove the right-continuity of $R$ in $-\mu_0$.
\begin{lemma}\label{continuitymu0}
\begin{equation}
\lim_{\mu \rightarrow -\mu_0}  R(\mu)=1.
\end{equation}
\end{lemma}
\begin{proof}
Suppose first that $R_G=1$. In that case, Proposition \ref{firstboundradius} yields    $R(\mu)=1, \; \forall \mu \in \mathbb{R}$ and thus the lemma is proved. Now suppose that $R_G>1$ (which implies that $m<\infty$).
Let $-\mu_0 <\mu<0$. We recall that for $s\in (q,R(\mu))$:
\begin{equation}
a^{\prime}(s)a(s)=-2\mu a(s)-2 \beta( G(s)-s). \tag{\ref{eqa}}
\end{equation}
For $1\leq s \leq R(\mu)$, we have $G(s)\geq s$, $a(s)\leq 0$ and thus $a^{\prime}(s)\geq -2\mu$. By integrating the previous equation, we obtain:
\begin{equation}\label{ineqanear1}
a(s)\geq -2\mu(s-1)+a(1).
\end{equation}
Knowing that $a(1)=Q^{\prime}(0,\mu)\leq 0$, we have that $s_1(\mu):=\frac{Q^{\prime}(0,\mu)}{2\mu}+1$
cancel the right hand side of \eqref{ineqanear1}. By the Intermediate Value Theorem, there is $s_2(\mu)\leq s_1(\mu)$ such that $a(s_2(\mu))=0$ and therefore 
\begin{equation}
1 \leq R(\mu)\leq s_1(\mu). \label{aleqrmuleqs1}
\end{equation} 
We have by Proposition \ref{lemcontinuQprilme0} that:
\begin{equation}
\lim_{\mu \rightarrow-\mu_0 }Q^{\prime}(0,\mu)=0,
\end{equation}
which means that: 
\begin{equation}
\lim_{\mu \rightarrow-\mu_0}s_1(\mu)=1.\label{aleqrmuleqs12}
\end{equation}
Equations \eqref{aleqrmuleqs1} and \eqref{aleqrmuleqs12} finally provide:
\begin{equation}
\lim_{\mu \rightarrow-\mu_0}R(\mu)=1.
\end{equation} 
\end{proof}
Note that \eqref{ineqanear1} implies that $Q^{\prime}(x_0(\mu),\mu)>-\infty$.
We now can more generally prove that the radius of convergence $R(\mu)$ is continuous on $R^{-1}([1,R_G))$.
\begin{lemma} \label{continuityR}
$\mu \mapsto R(\mu)$ is continuous on $R^{-1}([1,R_G))$.
\end{lemma}
Essentially, the key to the proof of Lemma \ref{continuityR} is Proposition \ref{lemcontinuQprilme0} and the continuity of the flow. We give this proof in Appendix. \\

 We will now tackle the last point of this section. As we mentioned in the introduction, whether $R(\mu)=R_G$ or $R(\mu)<R_G$ will be decisive to determine precisely the asymptotic behaviour of $q_n(x)$. We know that $R$ is non-decreasing and bounded by $R_G$ (we recall that $R_G$ can be infinite) and thus has a limit (not necessary finite) smaller or equal to $R_G$. We first show this limit is precisely $R_G$. After that, we will distinguish two cases which will allow us to determine whether there exists $\mu$ such that $R(\mu)=R_G$ or not.
 \begin{proposition}\label{propratteint}
 Let $r\in [1,R_G]$, if $\int_0^{r} G(x) \mathrm{d}x <+\infty $ then there exists $\mu_r$ such that $R(\mu_r)\geq r$.
 \end{proposition}
 Actually, the condition $\int_0^{r} G(x) \mathrm{d}x <+\infty $ is always satisfied for $r<R_G$, but we choose to formulate Proposition \ref{propratteint} in these terms to avoid repetitions.  
 \begin{proof}
Fix $r\in [1,R_G]$ and assume that $\int_0^{r} G(x) \mathrm{d}x <+\infty $. Furthermore, we suppose that for all $\mu\in \mathbb{R}$, $R(\mu)<r$. Let $\mu\geq 0$ and $s<R(\mu)$. By integrating \eqref{eqa}, we get:
\begin{align}
\frac{1}{2}\left(a^2(s,\mu)-a^2(1,\mu)\right)&=-2\mu \int_1^s a(u,\mu)\mathrm{d}u-2\beta\int_1^s\left(G(u)-u\right)\mathrm{d}u \nonumber \\
a^2(s,\mu)&=(Q^{\prime})^2(0,\mu)-4\mu \int_1^s a(u,\mu)\mathrm{d}u-4\beta\int_1^s\left(G(u)-u\right)\mathrm{d}u \nonumber \\
a^2(s,\mu)&\geq(Q^{\prime})^2(0,\mu)-4\beta\int_1^{r} \left(G(u)-u\right)\mathrm{d}u.
\end{align}
We can derive from \eqref{Qprime0} that: $\displaystyle \lim_{\mu \rightarrow +\infty} (Q^{\prime})^2(0,\mu)=+\infty$ and therefore for $\mu$ large enough there is $M>0$ such that: $a^2(s,\mu)>M$ for all $s<R(\mu)$. We cannot then have that $a^2(R(\mu),\mu)=0$, which is in contradiction with the fact that $R(\mu)<r\leq R_G$. 
\end{proof}
From Proposition \ref{propratteint}, we can  obviously derive the following corollary.
\begin{corollary}
$$\lim_{\mu\rightarrow +\infty}R(\mu)=R_G.$$
\end{corollary}

We now give the proof of Theorem \ref{threachornot} which is a criterion to know whether $R_G$ is reached by $R(\mu)$ or not. In Proposition \ref{propratteint} we have proved the implication:
\begin{equation}\label{eqimpliintGfinimplimucfini}
\int_{0}^{R_G} G(s) \mbox{d}s<+\infty \Longrightarrow\mu_c<\infty,
\end{equation} 
where $\mu_c$ is defined in \eqref{defmuc}. We prove in the following proposition the reciprocate implication.
\begin{proposition}\label{propimpli1}
If $\int_0^{R_G} G(s) \mathrm{d}s =+\infty$ then for all $\mu \in \mathbb{R}$, $R(\mu)<R_G$. 
\end{proposition}
Note that if $R_G=+\infty$, we have $\int_0^{R_G} G(s) \mathrm{d}s =+\infty$.
\begin{proof}
We suppose by contradiction that there exists $\mu>-\mu_0$ such that $R(\mu)=R_G$. Therefore, in this case by Theorem \ref{thcaractRwrtmu}, and by definition of $x_0$ in Proposition \ref{premdist}, we have that $x_0<0$, $Q$ is decreasing on $(x_0,0)$ and $Q(x_0)=R_G$. Let $x \in (x_0,0)$, we have by change of variable:
\begin{equation}\label{z2}
\int_0^x e^{2 \mu y }\left(G(Q(y))-Q(y)\right)\mathrm{d}y=\int_1^{Q(x)} e^{2 \mu Q^{-1}(s) }\frac{\left(G(s)-s\right)}{Q^{\prime}(Q^{-1}(s))}\mathrm{d}s.
\end{equation}
Moreover, \eqref{Qprimeequ} implies that:
\begin{equation}\label{z1}
Q^{\prime}(Q^{-1}(s))\geq Q^{\prime}(0)e^{-2\mu Q^{-1}(s)}.
\end{equation}
By introducing \eqref{z1} into \eqref{z2} we obtain:
\begin{equation}\label{z3}
\int_0^x e^{2 \mu y }\left(G(Q(y))-Q(y)\right)\mathrm{d}y\leq\int_1^{Q(x)} e^{4 \mu Q^{-1}(s) }\frac{\left(G(s)-s\right)}{Q^{\prime}(0)}\mathrm{d}s.
\end{equation}
We now suppose that $\mu\geq 0$. Using $\eqref{Qprimeequ}$, $\eqref{z3}$ and the fact that $Q^{-1}(s)\geq 0$ we obtain:
\begin{equation}
Q^{\prime}(x)\geq\left(Q^{\prime}(0)-\frac{2\beta}{Q^{\prime}(0)}\int_1^{Q(x)} \left(G(s)-s\right)\mathrm{d}s \right)e^{-2 \mu x}.
\end{equation}
We then have by comparison theorem:
\begin{equation}
\lim_{x \rightarrow x_0} Q^{\prime}(x)=+\infty,
\end{equation}
which is in contradiction with the assumption that $Q$ is decreasing on $(x_0,0)$. By Proposition \ref{premdist} we can conclude that there is $x_0<0$ such that $Q^{\prime}(x_0)=0$, $Q(x_0)<R_G$ and $Q$ decreasing on $(x_0,0)$, which implies that $R(\mu)=Q(x_0)$. We have supposed that $\mu \geq 0$ but since $R(\mu)$ is an non-decreasing function this result also holds for $\mu \in (-\mu_0,0]$.
\end{proof}
By gathering Proposition \ref{propimpli1} and \eqref{eqimpliintGfinimplimucfini}, we obtain Theorem \ref{threachornot}. We finish this section by giving an exhaustive description of $(Q(x_0),Q^{\prime}(x_0))$.
\begin{proposition}\label{Propimpl2}
Let $\mu \in (-\mu_0,+\infty)$.
\begin{enumerate}
\item If $\mu<\mu_c$ then $Q^{\prime}(x_0)=0$ and $Q(x_0)<R_G$;
\item if $\mu = \mu_c$ then $Q^{\prime}(x_0)=0$ and $Q(x_0)=R_G$;
\item if $\mu > \mu_c$ for all $x \geq x_0$, $Q^{\prime}(x_0)<0$ and $Q(x_0)=R_G$. 
\end{enumerate}
Furthermore, $\mu \mapsto R(\mu)$ is continuous on $\mathbb{R}$.
\end{proposition}
Note that if $R_G=1$ we are always in the third case of this proposition, and if $\int_0^{R_G} G(x) \mathrm{d}x =+\infty$ we are always in the first case. The continuity of $R$ have already been seen on $R^{-1}[1,R_G)$ and the first point is already known. We can reformulate what it remains to prove with the following lemma.
\begin{lemma} \label{lemma2annexe}
If $R_G>1$ and $\int_0^{R_G} G(x) \mathrm{d}x <+\infty$, $\mu_c$ is the unique $\mu$ such that $Q^{\prime}(x_0(\mu),\mu)=0$, $Q(x_0(\mu),\mu)=R_G$. Moreover, $R$ is continuous at $\mu_c$.
\end{lemma}
As for Lemma \ref{continuityR}, we give the proof of this lemma in Annexes.
\section{Case $R(\mu)<R_G$}
We chose to dedicate a section to this case because in this situation we can give an exact equivalent to $q_n(x)$ when $n$ tends to $+\infty$ by using complex analytical methods. The general idea to use complex analysis and more specifically the singularity analysis in this context is due to Maillard. Since the behaviour of $f_x$ near its singularities is different from that of $F_x$ in $\cite{maillard2013number}$, we nevertheless need to do some adjustments. We start with some notations and known results. The next lemma is Lemma 6.1 of \cite{maillard2013number}.
\begin{lemma}
\label{lem_delta}
The span of $Z_x$ and $\delta$ (the span of $G$) are equal.
\end{lemma}
We also gives an adaptation in our context of Lemma 6.2 of \cite{maillard2013number}.  Let $x_0$ be defined as in Proposition \ref{premdist} and $s_0=Q(x_0)=R(\mu)$. For $z\in \mathbb{C}$ and $r>0$, $D(z,r)$ will denote the open disc of center $z$ and radius $r$ and we fix $D=D(0,s_0)$ and $D_{\delta}=D(0,{s_0}^\delta)$. As usual, the frontier of a set $S$ is denoted by $\partial S$.
\begin{lemma}
\label{lem_h}
Fix $x>0$. If $\delta = 1$, then $f_x$ is analytical at every $s\in\partial D\backslash\{s_0\}$. 
If $\delta \geq 2$, then there exists an analytical function on $D_{\delta}$: $h_x$, such that:
\begin{equation}\label{defhx}
f_x(s) = sh_x(s^\delta), \; \forall s\in D.
\end{equation}
Moreover, $h_x$ is analytical at every $s\in\partial D_{\delta} \setminus \{s_0^{\delta}\}$.
\end{lemma}
The proof of the previous result can be adapted from \cite{maillard2013number} to our case with one exception. Indeed, we need to have $f_x(s_0)<\infty$, whose analogue is always satisfied in Maillard's case (since in his situation $s_0=1$ and $F_x(1)=1$) but which is not obvious in ours. However,  \eqref{linkomegQ} and the fact that $R(\mu)=Q(x_0)$ yield $\lim_{s\rightarrow s_0} f_x(s)=Q(x_0+x)<\infty$. Moreover, the coefficients of $f_x$ as a power series are non-negative. Therefore, by applying for instance the monotone convergence theorem we see that $f_x(s_0)=Q(x_0+x)<\infty$. \\

Finally, we state a reformulation in our framework of Corollary VI.1 of \cite{flajolet2009analytic}. For $z \in \mathbb{C}$, arg$(z)$ is chosen in $(-\pi,\pi]$. We call  a $\Delta$-domain, as in $\cite{flajolet2009analytic}$ and in $\cite{maillard2013number}$, a set defined for $\varphi \in (0,\pi/2)$, $s>0$ and $r>0$ by:
\begin{equation}\label{defdelta}
\Delta(\varphi,r,s) := \{z\in D(0,s+r)\setminus\{s\}: |\arg(z-s)|>\varphi\}.
\end{equation}
\begin{thmc}
Let $\varphi \in (0,\pi/2)$, $s>0$ and $r>0$. Let $\mathcal{H}(z):=\sum_{n=0}^{+\infty} H_n z^n$ be an analytical function on $\Delta(\varphi,r,s)$. If there exists $\alpha \in \mathbb{R}\setminus \mathbb{Z}_-$ such that:
$$\mathcal{H}(z) \underset{z \rightarrow r}{\sim} \frac{1}{(r-z)^{\alpha}}, \; z \in \Delta(\varphi,r,s)$$
then
$$H_n \underset{n \rightarrow +\infty}{\sim} \frac{n^{\alpha-1}}{r^{n+\alpha} \Gamma(\alpha) }.$$
\end{thmc}
To apply this theorem, we need the behaviour of $f_x$ when $\delta=1$ (resp. $h_x$, when $\delta\geq 2$) near its singularity $s_0$ (resp. $s_0^{\delta}$). 

Let us introduce the complex logarithm defined for $z \in \mathbb{C} \setminus \mathbb{R}^-$ by 
$$\log(z)=|z|+\mathbf{i}\arg(z)$$
and the complex square root defined on the same set by $$\sqrt{z}=e^{\frac{\log(z)}{2}}.$$ 
\begin{lemma}\label{asmpomermuinfrg}
For each $x>0$, there exists $r_{1,x}>0$ such that $s \mapsto f_x(s)$ is analytical on $D(s_0,r_{1,x}) \setminus (s_0,+\infty)$, and for $s$ in this set we have:
\begin{equation}
f^{\prime}_x(s) \underset{s\rightarrow s_0}{\sim} \frac{-Q^{\prime}\left(x_0+x\right)}{2\sqrt{\beta(s_0-s)(G(s_0)-s_0)}}.\label{asympfprimlemma43}
\end{equation}
Similarly, when $\delta\geq 2$, for each $x>0$, there exists $r_{\delta,x}>0$ such that $s \mapsto h_x(s)$ is analytical on $D(s^{\delta}_0,r_{\delta,x}) \setminus (s^{\delta}_0,+\infty)$, and for $s$ in this set we have:
\begin{equation}
h^{\prime}_x(s)\underset{s\rightarrow {s_0}^{\delta}}{\sim}-\frac{Q^{\prime}(x_0+x)}{2\sqrt{\beta \delta s_0^{\delta+1}(s^\delta_0-s)(G(s_0)-s_0)}}.\label{asymphprimlemma43}
\end{equation}
\end{lemma}
\begin{proof}
To prove the analyticity of $f_x$, we will use and extend in a complex sense Equation \eqref{linkomegQ}. In this equation, the inverse function $Q^{-1}$ is only defined on $J$ (defined in Proposition \ref{premdist}), that is why we will find an analytical function defined near (in a sense we will precise below) $s_0$ which coincides with $Q^{-1}$ on $J$.  \\

By Proposition \ref{Propimpl2}, when $\mu<\mu_c$, $Q^{\prime}(x_0)=0$. Moreover, Equation \eqref{KKPx} implies that $Q^{\prime \prime}(x_0)<0$. Since $Q(x_0)<R_G$, $Q$ admits an analytical extension near $x_0$ by Proposition \ref{extension}. Thus in the complex plane near $x_0$ we have:
\begin{equation}
Q(z)=Q(x_0)+(z-x_0)^2\frac{Q^{\prime \prime}(x_0)}{2}+\underset{z\rightarrow x_0}{o}\left(\left(z-x_0\right)^2\right).
\end{equation}
The function $Q$ is analytical on an neighbourhood of $x_0$, which is a zero of order $2$ of $Q(z)-Q(x_0)$. Theorem 10.32 of \cite{rudin1987real} thus ensures that there exists $r_1>0$ such that on $D(x_0,r_1)$, there exists an analytical invertible function $\psi : D(x_0,r_1) \rightarrow \mathbb{C}$ such that:
\begin{equation}\label{dlx0}
Q(z)=Q(x_0)+\frac{Q^{\prime \prime}(x_0)}{2}\psi(z)^2.
\end{equation}
Note that Equation \eqref{dlx0} implies that for $z\in J\cap D(x_0,r_1)$ we have $\psi\left(z\right)\in \mathbb{R}\cup \mathbf{i}\mathbb{R}$. More precisely, $\psi\left(z\right)\in \mathbb{R}$, because if $\psi(z)\in \mathbf{i}\mathbb{R}$, we would have $Q(z)>Q(x_0)$, which is in contradiction with the definition of $J$. Furthermore, the Intermediate Value Theorem and the fact that $Q(z)<Q(x_0), \forall z \in J$ implies that if there exists $z_0\in J\cap D(x_0,r_1)$ such that $\psi\left(z_0\right)>0$ then for all $z\in J\cap D(x_0,r_1)$, $\psi\left(z\right)>0$ . Finally, since we can substitute $-\psi$ for $\psi$ in \eqref{dlx0}, we can chose $\psi$ such that $\psi>0$ on $J\cap D(x_0,r_1)$. \\

We now fix $s\in[0,s_0) \cap Q\left(D(x_0,r_1)\right)$ (note that $Q$ is analytic, and thus it is an open application, which implies that this intersection is not empty). By choosing $z=Q^{-1}(s)$ in \eqref{dlx0} and using \eqref{KKPx}, we get:
\begin{equation}\label{Qinvparrappsi}
\psi^{-1}\left(\sqrt{\frac{s_0-s}{\beta(G(s_0)-s_0)}}\right)=Q^{-1}(s).
\end{equation}
By considering the complex square root, we can define
\begin{equation}\label{defeta}
\eta(s):=\psi^{-1}\left(\sqrt{\frac{s_0-s}{\beta(G(s_0)-s_0)}}\right), \; \forall s \in Q\left(D\left(x_0,r_1\right)\right)\setminus (s_0,+\infty).
\end{equation} 
We now  recall that for $s \in [0,s_0)$ and $x>0$:
\begin{equation}\tag{\ref{linkomegQ}}
f_x(s)=Q(Q^{-1}(s)+x).
\end{equation}
By Proposition \ref{extension}, there exists $r_2>0$, such that $Q$ is analytical on $D(x_0+x,r_2)$. Furthermore, Equations \eqref{linkomegQ}, \eqref{Qinvparrappsi} and \eqref{defeta} yield $r_3>0$ small enough such that $f_x(s)$ and $Q(\eta(s)+x)$ exist for $s \in [0,s_0) \cap D(s_0,r_3)$ and coincide. Formally, we choose $0<r_3\leq r_1$ such that: 

\begin{equation}\label{defr3}
  \left\{
      \begin{aligned}
       D(s_0,r_3) \subset Q\left(D(x_0,r_1)\right)\\
        Q^{-1}\left(D(s_0,r_3)\cap [0,s_0)\right)+x \subset D(x_0+x,r_2)\\
        \eta\left(D(s_0,r_3)\setminus (s_0,+\infty)\right)+x \subset D(x_0+x,r_2).\\
        \end{aligned}
    \right.
\end{equation}
Moreover, $D(s_0,r_3)  \setminus (s_0,+\infty)$ is an open connected set and $[0,s_0) \cap D(s_0,r_3)$ is a subset of it with an accumulation point. Therefore, $f_x(s)$ admits an unique analytical extension on $D(s_0,r_3)\setminus (s_0,+\infty)$ which is $Q\left(\eta(s)+x\right)$. Thus, we have that, for $s \in D(s_0,r_3)\setminus (s_0,+\infty)$:
\begin{equation}
f^{\prime}_x(s)=-\frac{{(\psi^{-1})}^{\prime}\left(\sqrt{\frac{s_0-s}{\beta(G(s_0)-s_0)}}\right)}{2\sqrt{\beta(s_0-s)(G(s_0)-s_0)}}Q^{\prime}\left(\psi^{-1}\left(\sqrt{\frac{s_0-s}{\beta(G(s_0)-s_0)}}\right)+x\right),
\end{equation}
which implies that:
\begin{equation}\label{finproofequivalfx}
f^{\prime}_x(s) \underset{s\rightarrow s_0}{\sim} -\frac{{(\psi^{-1})}^{\prime}\left(0\right)}{2\sqrt{\beta(s_0-s)(G(s_0)-s_0)}}Q^{\prime}\left(x_0+x\right).
\end{equation}
To get the value of ${(\psi^{-1})}^{\prime}\left(0\right)$ we first differentiate $\eqref{dlx0}$ with respect to $z$:
\begin{equation}\label{dveuler0}
Q^{\prime}(z)=Q^{\prime \prime}(x_0)\psi^{\prime}(z)\psi(z).
\end{equation}
Furthermore, Taylor's formula yields:
\begin{equation}\label{dveuler1}
  \left\{
      \begin{aligned}
       Q^{\prime}(z)=Q^{\prime \prime}(x_0)(z-x_0)+o(z-x_0)\\
       \psi(z)=\psi^{\prime}(x_0)(z-x_0)+o(z-x_0).\\
        \end{aligned}
    \right.
\end{equation}
Equations \eqref{dveuler0} and \eqref{dveuler1} provide $(\psi^{\prime})^2(x_0)=1$. We have chosen $\psi$ such that $\psi>0$ on $(x_0,x_0+r_1)$ and thus $\psi^{\prime}(x_0)=1$. As a consequence, ${(\psi^{-1})}^{\prime}\left(0\right)=1$, which yields \eqref{asympfprimlemma43}. \\

We can derive from the results on $f_x$ the analogous results on $h_x$. Let us define on $\mathbb{C}\setminus \mathbb{R}^-$ the function $z \mapsto \sqrt[\delta]{z}:=e^{\frac{\log(z)}{\delta}}$. Lemma \ref{lem_h} yields:

\begin{equation}\label{hxenfonctionfx}
h_x(s)=\frac{f_x(\sqrt[\delta]{s})}{\sqrt[\delta]{s}}, \forall s \in D(0,s^{\delta}_0) \setminus \mathbb{R}^-.
\end{equation}
Besides, since there exists $0<r_{1,x}<s_0$ such that $f_x$ is analytic on $D(s_0,r_{1,x})\setminus (s_0,+\infty)$, we can show after some change of variable that there exists $0<r_{\delta,x}<s^{\delta}_0$ such that the right term of $\eqref{hxenfonctionfx}$ is analytic on $D(s^{\delta}_0,r_{\delta,x})\setminus(s^{\delta}_0,+\infty)$. The function $h_x$ and the right term of $\eqref{hxenfonctionfx}$ coincide on a open subset of the connected set $D(s^{\delta}_0,r_{\delta,x})\setminus(s^{\delta}_0,+\infty)$ and thus $h_x$ has an analytic extension on $D(s^{\delta}_0,r_{\delta,x})\setminus(s^{\delta}_0,+\infty)$. 

Let us turn to the behaviour of $h_x$ near $s_0^{\delta}$. By differentiating $\eqref{hxenfonctionfx}$ with respect to $s$, we get:
\begin{equation}\label{hxprimfxpri}
h^{\prime}_x(s)=\frac{f^{\prime}_x(\sqrt[\delta]{s})\sqrt[\delta]{s}-f_x(\sqrt[\delta]{s})}{\delta s\sqrt[\delta]{s}}.
\end{equation}
Equation \eqref{finproofequivalfx} yields: 
\begin{align}
f^{\prime}_x(\sqrt[\delta]{s}) &\underset{s\rightarrow s^{\delta}_0}{\sim} -\frac{Q^{\prime}\left(x_0+x\right)}{2\sqrt{\beta\left(\sqrt[\delta]{s^{\delta}_0}-\sqrt[\delta]{s}\right)(G(s_0)-s_0)}} \nonumber \\
&\underset{s\rightarrow s^{\delta}_0}{\sim} -\frac{Q^{\prime}\left(x_0+x\right)\sqrt{\delta}}{2\sqrt{\beta(s^{\delta}_0-s)(s_0^{1-\delta})(G(s_0)-s_0)}}, \label{fprims0delta}
\end{align}
since:
$$
\lim_{s \rightarrow s^{\delta}_0} \frac{\sqrt[\delta]{s^{\delta}_0}-\sqrt[\delta]{s}}{s^{\delta}_0-s}=\frac{1}{\delta}{(s^{\delta}_0)}^{\frac{1}{\delta}-1}.
$$
Moreover, we have seen that $f_x(s_0)<+\infty$. Therefore, by introducing \eqref{fprims0delta} into \eqref{hxprimfxpri}, we get \eqref{asymphprimlemma43}.
\end{proof}
The end of the proof of Theorem $\ref{asymptotic}$ is almost identical to that from Theorem 1.2 in $\cite{maillard2013number}$ up to the fact that the asymptotic is not the same. 
\begin{proof}[Proof of Theorem \ref{asymptotic}]
Let us just give the principal steps when $\delta=1$. Let us take $\varphi_0\in [0,\pi/2)$ small enough such that $s_0e^{\mathbf{i}\varphi_0} \in D(s_0,r_x)$. By Lemma \ref{lem_h}, for each $s\in S:=\lbrace s_0e^{\mathbf{i}\varphi}, \varphi \in [\varphi_0, 2\pi-\varphi_0] \rbrace$, there exists $\mathcal{R}_s>0$ such that $f_x$ admit an analytical extension on $D(s,\mathcal{R}_s)$. Furthermore, $S$ is compact. Hence, there exist $k\in \mathbb{N}$ and $(z_i) \in S^k$ such that:
$$S \subset T:=\underset{i \in \lbrace 1,...k \rbrace}{\cup} D(z_i,\mathcal{R}_{z_i}).$$
Since $f_x$ is analytical on $T\cup D(0,s_0)$ and by Lemma \ref{asmpomermuinfrg}  on $D(s_0,r_x) \setminus (s_0, + \infty)$, we can find $\epsilon>0$ such that $f_x$ is analytical on $\mathcal{D}:=D(0,s_0+\epsilon)\setminus (s_0,+\infty)$.
We now can apply Corollary VI.1 of \cite{flajolet2009analytic}. Indeed, $\mathcal{D}$ contains a $\Delta$-domain which satisfies the assumptions of this Corollary and we precisely know the asymptotic of $f^{\prime}_x$ near $s_0$ by Lemma \ref{asmpomermuinfrg}. Since the $i$th coefficient of $f^{\prime}_x$ as a power series is $(i+1)q_{i+1}(x)$, we get:
$$(i+1)q_{i+1}(x) \underset{i\rightarrow + \infty}{\sim} \frac{-Q^{\prime}(x_0+x)}{2\sqrt{\beta(G(R(\mu))-R(\mu))}}\times \frac{1}{R(\mu)^{i+\frac{1}{2}}\sqrt{i}\Gamma(\frac{1}{2})},$$
which is Theorem \ref{asymptotic} when $\delta=1$.\\

When $\delta \geq 2$, we can also find a good $\Delta$-domain on which $h_x$ is analytical. Furthermore, by definition of $h_x$ \eqref{defhx}, the $i$th coefficient of $h^{\prime}_x$ is $(i+1)q_{\delta(i+1)+1}$. Therefore Corollary VI.1 of \cite{flajolet2009analytic} and Lemma \ref{asmpomermuinfrg} similarly provide \eqref{asymptoticcontenu}.

\end{proof}
\section{Case where $R(\mu)=R_G$}
Let $\mu \in \mathbb{R}$. In the previous section, we handled the case $R(\mu)<R_G$ which includes the case $\mu_c=+\infty$. We now consider the case where $\mu_c<+\infty$ ($\mu_c$ can be equal to $-\infty$) and $\mu \geq \mu_c$. This is equivalent to the case $\int_0^{R_G}G(s)\mathrm{d}s<+\infty$ and $R(\mu)=R_G$. The asymptotic behaviour of $q_k(x)$ when $k$ tends to $+\infty$ was obtained when $R(/mu)<R_G$ by studying $f_x$ near its radius of convergence. Since $R(\mu)=R_G$, it is not possible anymore to extend $f_x$ to a $\Delta$-domain analytically as in the previous section. However, in some case, the behaviour of $G$ as a real function near $R_G$ gives us weaker results. \\

Suppose first that $\mu=\mu_c$.  As a first step, we will give the behaviour of $f_x(s)$, when $s\rightarrow {R_G}^-$, $s\in \mathbb{R}$.
\begin{lemma}\label{lemmafxprim}
If $\mu=\mu_c$, we have:
\begin{equation}\label{asympfxcrit}
f^{\prime}_x(s)\underset{s\rightarrow R_G}{\sim}- \frac{Q^{\prime}(x+x_0)}{2\sqrt{\beta \int_s^{R_G} (G(u)-u) \mathrm{d}u}}, \; s \in (q,R_G).
\end{equation}
\end{lemma}
We cannot use the exact same arguments as in the proof of Lemma \ref{asmpomermuinfrg}  since $Q$ is not analytical at $x_0$ anymore.
\begin{proof}
Suppose that $\mu=\mu_c$ and fix $s<R_G$.
As a consequence of Proposition \ref{Propimpl2}, when $s$ tends to $R_G$, $-2\mu_c a(s)$ tends to $0$. Moreover, if $R_G=1$, then $\forall \mu \in \mathbb{R}, R(\mu)=R_G$ and therefore, by definition of $\mu_c$, $\mu_c=-\infty$. Here, since we have supposed $\mu_c=\mu\in\mathbb{R}$, $\mu_c>-\infty$ and thus $R_G>1$. Hence, we have that $2\beta\left(G(s)-s\right)\geq C$, where $C>0$, for $s$ in a neighbourhood of $R_G$. Therefore, \eqref{eqa} yields:
\begin{equation}\label{aprimcri2}
a^{\prime}(s)a(s)\underset{s\rightarrow R_G}{\sim}-2\beta\left(G(s)-s\right).
\end{equation}
We have that $\int_{0}^{R_G} G(u)\mathrm{d}u<\infty$, and therefore by integration:
\begin{equation}\label{aprimcri30}
-\frac{1}{2}a^2(s)\underset{s\rightarrow R_G}{\sim}-\int_{s}^{R_G}2\beta\left(G(u)-u\right)\mathrm{d}u.
\end{equation}
Since $a \leq 0$, \eqref{aprimcri30} implies that:
\begin{equation}\label{aprimcri3}
a(s)\underset{s\rightarrow R_G}{\sim}-2\sqrt{\beta \int_s^{R_G} (G(u)-u) \mathrm{d}u}.
\end{equation}
On the other hand, by diffrentiating \eqref{linkomegQ}, we obtain the following equation which is called the forward Kolmogorov equation (see for instance Section 3. chapter III of \cite{athreya1972branching} for more information in the general case)
\begin{equation}\label{kolfor}
\partial_x f_x(s)=a(s)f^{\prime}_x(s).
\end{equation}
Moreover, For $x$ fixed, \eqref{defa} yields:
\begin{equation}\label{partialxf}
\partial_x f_x(s)=Q^{\prime}(Q^{-1}(s)+x)\underset{s\rightarrow R_G}{\sim} Q^{\prime}(x_0+x).
\end{equation}
Combining \eqref{aprimcri3}, \eqref{kolfor} and \eqref{partialxf} we get the result.
\end{proof}
Observe that if $G(R_G)<\infty$, 
\begin{equation}
f^{\prime}_x(s)\underset{s\rightarrow R_G}{\sim}- \frac{Q^{\prime}(x+x_0)}{2\sqrt{\beta(R_G-s)(G(R_G)-R_G)}}.
\end{equation}
In this case, $f^{\prime}_x$ has the same kind of asymptotic near its radius of convergence as in the previous section and thus it is likely we will have the same kind of asymptotic for $q_i(x)$. However, the asymptotic of $f_x$ is, this time, only in the real sense. Nevertheless, a Tauberian theorem can be used to obtain a rougher description of the large $i$ behaviour of $q_i(x)$. In what follows, we consider a case a slightly more general than $G(R_G)<\infty$ by supposing that:
\begin{equation}\label{asmpGforthcritical}
G(s)-s\underset{s \rightarrow R_G}{\sim} C(R_G-s)^{-\alpha},
\end{equation}
where $C>0$, $\alpha \in [0,1)$ and $s\in(q,R_G)$. Since we are in the case $\mu_c<\infty$, we necessarily have that $\int_0^{R_G}G(s)\mathrm{d}s<+\infty$, which explains why $\alpha$ must be in $[0,1)$. Note that if $\alpha=0$, Equation \eqref{asmpGforthcritical} is equivalent to $G(R_G)<\infty$ and in this case $C=G(R_G)-R_G$, whereas if $\alpha>0$ we can replace $ G(s)-s$ by $G(s)$ in  \eqref{asmpGforthcritical}. 
\begin{proposition}\label{thasymptcritassump1}

If Equation \eqref{asmpGforthcritical} holds then:
\begin{equation}\label{qdeltamumucinftaub}
\sum_{i=n}^{+\infty}q_{i}(x)R_G^i\underset{n\rightarrow+\infty}{\sim}\frac{-AQ^{\prime}(x_0+x)}{n^{\frac{1+\alpha}{2}}},
\end{equation}
where $A:=\frac{\sqrt{(1-\alpha)R_G^{1+\alpha}}}{(1+\alpha)\sqrt{\beta C }\Gamma\left(\frac{1-\alpha}{2}\right)}$.  
\end{proposition}
We can observe that if we give an equivalent of  $\sum_{i=n}^{+\infty}q_{i}(x)R_G^i$ when $\mu<\mu_c$ with the help of Theorem \ref{asymptotic}, we get the same result as in Proposition \ref{thasymptcritassump1} when $\alpha=0$.
\begin{proof}
Combining Equation \eqref{asympfxcrit} and \eqref{asmpGforthcritical} we get:
\begin{equation}\label{asympfxcrit2}
f^{\prime}_x(s)\underset{s\rightarrow R_G}{\sim}- \frac{Q^{\prime}(x+x_0)\sqrt{1-\alpha}}{2\sqrt{\beta C (R_G-s)^{1-\alpha}}}.
\end{equation}
Let us rescale $f^{\prime}_x$ by defining $b_x(s)=f_x^{\prime}(sR_G)$. Since the radius of convergence of $f_x$ is $R_G$, the radius of convergence of $b_x$ is $1$. By rescaling \eqref{asympfxcrit2} we obtain:
\begin{equation}\label{asympfxcritsc21}
b_x(s)\underset{s\rightarrow R_G}{\sim}- \frac{Q^{\prime}(x+x_0)\sqrt{1-\alpha}}{2\sqrt{\beta C R_G^{1-\alpha}(1-s)^{1-\alpha}}}.
\end{equation}
Using Theorem 5. Chapitre XIII Section 5 of \cite{feller1971introduction} we get:
\begin{equation}\label{asympfxcritsc3}
V(n):=\sum_{i=0}^{n-1}(i+1)q_{i+1}(x)R_G^i\underset{n \rightarrow +\infty}{\sim}- \frac{n^{\frac{1-\alpha}{2}}Q^{\prime}(x+x_0)}{\sqrt{1-\alpha}\Gamma(\frac{1-\alpha}{2})\sqrt{\beta C R_G^{1-\alpha}}}.
\end{equation}
The definition of $V$ yields:
\begin{equation}\label{qR_gcrit1}
q_i(x)R_G^i=\frac{(V(i)-V(i-1))R_G}{i}.
\end{equation}
We recall that $f_x(R(\mu))=f_x(R_G)<+\infty$. Therefore, by using \eqref{asympfxcritsc3} and \eqref{qR_gcrit1} and the fact that the terms of each of the series which follow are all positive we get:
\begin{align}
\sum_{i=n}^{\infty}q_i(x)R_G^i & \quad = \quad \quad \sum_{i=n}^{\infty}\frac{(V(i)-V(i-1))R_G}{i}  \nonumber\\
& \quad = \quad \quad \sum_{i=n}^{\infty}\frac{V(i)R_G}{i(i+1)}-\frac{V(n-1)R_G}{n}  \label{asymptotRGcrit2ndligne}\\
&\underset{n \rightarrow +\infty}{\sim}- \frac{Q^{\prime}(x+x_0)R_G}{\sqrt{1-\alpha}\Gamma(\frac{1-\alpha}{2})\sqrt{\beta C R_G^{1-\alpha}}}\left[\sum_{i=n}^{\infty}\frac{1}{i^{\frac{3+\alpha}{2}}}-\frac{1}{n^\frac{1+\alpha}{2}}\right]\nonumber\\
&\underset{n \rightarrow +\infty}{\sim}- \frac{AQ^{\prime}(x+x_0)}{n^{\frac{1+\alpha}{2}}}, \label{asympqR_gcrit2}
\end{align}
where $A:=\frac{\sqrt{(1-\alpha)R_G^{1+\alpha}}}{(1+\alpha)\sqrt{\beta C }\Gamma\left(\frac{1-\alpha}{2}\right)}$. Equation \eqref{asymptotRGcrit2ndligne} is obtained by integration by parts and Equation \eqref{asympqR_gcrit2} by a classical comparison between sums and integrals. 
\end{proof}

The influence of $G$ on the $q_i(x)$ seems difficult to understand when $\mu=\mu_c$. We will now see that in the case where $\mu>\mu_c$, there exists a stronger link between the $p_{i}$ and the $q_{i}(x)$. As a first step we give a result which do not require any specific knowledge on $G$.
\begin{proposition} Let $k \in \mathbb{N}$, if $\mu>\mu_c$ we have:
\begin{equation}
\sum_{i=0}^{\infty} q_i(x)R_G^i i^{k+2} <\infty \Leftrightarrow \sum_{i=0}^{\infty} p_iR_G^i i^{k}<\infty .
\end{equation}
\end{proposition}
\begin{proof}
We begin by prove this result for $k=0$. First suppose that 
\begin{equation}\label{assumpG1}
G(R_G)=\sum_{i=0}^{\infty} p_iR_G^i <+\infty.
\end{equation}
We recall from \eqref{eqa} that: 
$$
a^{\prime}(s)a(s)=-2\mu a(s)-2\beta\left(G(s)-s\right) \forall s\in (q,Q(x_0)).
$$
Since $\lim_{s\rightarrow R_G} a(s)\in \mathbb{R}^*$,  Equations \eqref{defa} and \eqref{assumpG1} imply that $\lim_{s\rightarrow R_G} a^{\prime}(s)<\infty$. As above, by using Kolmogorov equations we get:
\begin{align}
a(s)f^{\prime}_x(s)&=a(f_x(s)) \label{doubleKolsurcrit} \\
a^{\prime}(s)f^{\prime}_x(s)+a(s)f^{\prime \prime}_x(s)&=f^{\prime}_x(s)a^{\prime}(f_x(s))\label{doubleKolsurcrit2}.
\end{align}
Equation \eqref{doubleKolsurcrit} implies that $\lim_{s\rightarrow R_G} f^{\prime}_x(s)<\infty$ and \eqref{doubleKolsurcrit2} similarly implies that $\lim_{s\rightarrow R_G} f^{\prime \prime}_x(s)<\infty$. The coefficients of the power series $f^{\prime \prime}_x(s)$ are positive, therefore $f^{\prime \prime}_x(R_G)<\infty$, which is equivalent to 
\begin{equation}
\sum_{i=0}^{\infty} q_i(x)R_G^i i^{2} <\infty.
\end{equation}
If we now suppose that for $k\in \mathbb{N}$, $\sum_{i=0}^{\infty} p_iR_G^i i^{k}=\infty$, we can  similarly prove that 
\begin{equation}
\sum_{i=0}^{\infty} q_i(x)R_G^i i^{k+2} =\infty.
\end{equation}
The general case can be proved by induction. The proof is almost identical to the case $k=0$, we differentiate \eqref{doubleKolsurcrit} $k+1$ times, and use the induction hypothesis to determinate what is finite or not.
\end{proof}
This result is pretty weak, but informally it shows that a link exists between $p_i$ and $q_{i}(x)i^2$. Once again, with more specific assumptions on $G$ we can give a more accurate result on $q_i(x)$ which confirms this link.
\begin{proposition}
Let $m\in \mathbb{N}$ and suppose that $G^{(m)}(s)\underset{s\rightarrow R_G}{\sim} C(R_G-s)^{-\alpha}$, with $C>0$ , $\alpha\in(0,m+1)$ and $s\in (q,R_G)$. Suppose that $\mu>\mu_c$, then for $t\in \mathbb{R}$ such that $t<m+2-\alpha$, we have:
\begin{align}
\sum_{i=n}^{+\infty}q_{i}(x)R_G^i i^t &\underset{n\rightarrow+\infty}{\sim}\frac{K_1Q^{\prime}(x+x_0)}{n^{m+2-t-\alpha}},  \nonumber \\
&\underset{n\rightarrow+\infty}{\sim}K_2 Q^{\prime}(x+x_0)\sum_{i=n}^{+\infty}p_iR_G^i i^{t-2},
\end{align}
where $K_1=\frac{2\beta C R^{m+2-\alpha}_G}{\Gamma(\alpha)(m+2-t-\alpha)(Q^{\prime})^3(x_0)}$ and $K_2=\frac{2\beta R^2_G}{(Q^{\prime})^3(x_0)}$.
\end{proposition}
We recall that $\int_0^{R_G} G(s)\mathrm{d}s<+\infty$ in this case, which explains that $\alpha<m+1$. 
\begin{proof}
Since the proof of this proposition is very close to that of Proposition \ref{thasymptcritassump1} we will skip some details. For $0 \leq k\leq m$ and $s\in (q,R_G)$ we get by differentiating $k$ times \eqref{eqa} that:
\begin{equation}
\sum_{i=0}^k \binom{k}{i} a^{(i+1)}(s)a^{(k-i)}(s)=-2\mu a^{(k)}(s)-2\beta\left(G^{(k)}(s)-\phi^{(k)}(s)\right), \label{differaktimes}
\end{equation}
where $\phi(x)=x$. Thanks to \eqref{differaktimes}, we can show by induction that for all $k\leq m$, $\lim_{s \rightarrow R_G} a^{(k)}(s)<+\infty$ and thus that:
\begin{equation}
a^{(m+1)}(s)Q^{\prime}(x_0)\underset{s\rightarrow R_G}{\sim} -\frac{2C\beta}{(R_G-s)^{\alpha}}. \label{aderivmplusone}
\end{equation}
By differentiating $2$ times Equation \eqref{linkomegQ} we get for $q<s<R_G$:
\begin{equation}
f_x^{\prime \prime}(s)=\frac{Q^{\prime \prime}(Q^{-1}(s)+x)-a^{\prime}(s)Q^{\prime}(Q^{-1}(s)+x)}{a^2(s)}. \label{linkomegafxseconde}
\end{equation}
If we differentiate again $m$ times Equation \eqref{linkomegafxseconde}, we can show that
\begin{equation}
f_x^{(m+2)}(s)\underset{s\rightarrow R_G}{\sim} \frac{-a^{(m+1)}(s)Q^{\prime}(x_0+x)}{(Q^{\prime})^2(x_0)},\label{fxmplus2one}
\end{equation}
since the others terms involve at most the $m$-th derivative of $a$ and thus are finite near $R_G$. Combining \eqref{aderivmplusone} and \eqref{fxmplus2one}, we obtain:
\begin{equation}
f_x^{(m+2)}(s)\underset{s\rightarrow R_G}{\sim} \frac{2C\beta Q^{\prime}(x_0+x)}{(Q^{\prime})^3(x_0)(R_G-s)^{\alpha}}.\label{fxmplus2two}
\end{equation}
We conclude by using the Tauberian Theorem as in the proof of Proposition \ref{thasymptcritassump1}.
\end{proof}
We recall that the probability to have $i$ particles on the barrier at the extinction time and only one split before the extinction calculated in \eqref{dxzxzetax} is of order $p_i/i^2$. This fact and the two previous propositions lead us to think that when $\mu>\mu_c$, there is also a change of regime for the number of divisions $D_x$ before extinction and we can conjecture, for instance, that the radius of convergence of the generating function (on the extinction event) of $D_x$ is infinite. In any case, the law of $D_x$ should cast light on the change of behaviour of $Z_x$ when $\mu>\mu_c$ or when $\mu=\mu_c$. Unfortunately, the study of $D_x$ seems much more difficult to that of $Z_x$, since its generating function does not solve simple equations.

\begin{appendix}
\section{Appendix}
\subsection{Proof of Lemma \ref{continuityR}}
\begin{proof}
Let $\mu_1\in \mathbb{R}$ such that $R(\mu_1) \in (1,R_G)$. Roughly speaking, to prove the continuity of $R$ in $\mu_1$ we will consider a domain around the trajectory of $X(\cdot ,\mu_1)$ (defined in \eqref{defX}) narrow enough such that for $\mu$ close enough to $\mu_1$, the trajectory of $X(\cdot ,\mu)$ is in this domain and cut the x axis near $Q(x_0,\mu_1)$. \\
We have proved that there exists $x_0 \in \mathbb{R}^-$ such that $X(x_0,\mu_1)=(R(\mu_1),0)$ and $Q$ is decreasing on $(x_0,+\infty)$. Furthermore, if we define $x_l$ as in Proposition $\ref{extension}$ and if $R(\mu_1)<R_G$, we necessarily have that $x_0>x_l$. Indeed, since in this case $Q^{\prime}(x_0)=0$ and $Q(x_0)<R_G$, Equation $\eqref{boundxl}$ implies that $x_0\neq x_l$. \\
Let $0<\epsilon<R_G-R(\mu_1)$, there exists $0<\eta_0<x_l-x_0$ such that:
\begin{equation}
 |Q(x,\mu_1)-R(\mu_1)|+|Q^{\prime}(x,\mu_1)|<\epsilon, \; \forall x \in (x_0-\eta_0,x_0+\eta_0) \label{voiseta0}
 \end{equation}
 and
 \begin{equation}
Q^{\prime}(x,\mu_1)\neq 0, \; \forall x \in (x_0-\eta_0,x_0+\eta_0) \setminus \lbrace x_0 \rbrace .\label{voiseta02}
 \end{equation}
Equation \eqref{voiseta0} comes from the continuity of $X$ with respect to $x$ and \eqref{voiseta02} is a consequence of Cauchy-Lipschitz Theorem.
Let $(x_1,x_2) \in (x_0,x_0+\eta_0)  \times (x_0-\eta_0,x_0)$. Since the function $Q^{\prime}$ is negative and continuous on $[x_2,0]$, there exists $\epsilon_2>0$ such that $Q^{\prime}(x)<-\epsilon_2, \forall x \in [x_2,0]$. We now  define $D$ the domain of $\mathbb{R}^2$:
\begin{equation}
D=\bigcup_{x\in [0,x_1]} B(X(x,\mu_1),\min \lbrace \epsilon, \epsilon_2 \rbrace),
\end{equation}
where $B(x,r)$ is the open ball for the norm $||\cdot ||_1$ of radius $r$ centered at $x$.
Let $I_0$ be an open bounded interval such that $\mu_1 \in I_0 \subset \left(-\mu_0,+\infty\right)$ and for $\xi=(\xi_1,\xi_2)\in \mathbb{R}^2$, $\psi(\cdot ,\xi,\mu)$ is defined as the maximal solution of $\eqref{deff}$ such that $\psi(0,\xi,\mu)=\xi$. The function $\Gamma$, defined in \eqref{deff}, is continuous and uniformly Lipschitz with respect to $((x,y),\mu)$ on $D\times I$. Therefore by Theorem 7.4 Section 7 Chapter 1 of \cite{coddington1955theory}, there exists $\delta>0$ such that if $(\xi,\mu) \in V$, where $V$ is defined by:
\begin{equation}\label{neighbourinitKPPpar}
V=\lbrace \left(\xi=(\xi_1,\xi_2),\mu\right), |\xi_1-Q(0,\mu_1)|+|\xi_2-Q^{\prime}(0,\mu)|+|\mu-\mu_1|<\delta \rbrace
\end{equation}
then for all $x\in [0,x_1]$, $\psi(x,\xi,\mu)$ is defined and $\psi(x,\xi,\mu)\in D$. Moreover, $\psi$ is continuous in $[0,x_1] \times V$. We have shown in Proposition \ref{lemcontinuQprilme0} that $\mu \mapsto (Q(0,\mu),Q^{\prime}(0,\mu))$ is continuous. Therefore, there exists $\eta_1>0$ such that if $\mu\in I_0$ satisfies $|\mu-\mu_1|<\eta_1$ then $(Q(0,\mu),Q^{\prime}(0,\mu),\mu) \in V$. Since $(Q(x,\mu),Q^{\prime}(x,\mu))=\psi(x,(Q(0,\mu)),Q^{\prime}(0,\mu)),\mu)$, we have that $X$ is continuous on $V_2:=[0,x_1]\times (\mu_1-\eta_1,\mu_1+\eta_1)$ and $X(V_2) \subset D$. \\
As an easy consequence of the fact that $Q^{\prime \prime}(x_0,\mu_1)<0$ and of \eqref{voiseta02} we have $Q^{\prime}(x_1,\mu_1)>0$. 
Therefore, by continuity of $X$, there exists $\eta_2<\eta_1$ such that if $|\mu-\mu_1|<\eta_2$ then $Q^{\prime}(x_1,\mu)>0$ and $Q^{\prime}(x_2,\mu)<0$. Hence, by the Intermediate Value Theorem, there exists $x_3\in (x_1,x_2)$ such that $Q^{\prime}(x_3,\mu)=0$. Furthermore, by the definition of $D$ and by \eqref{voiseta02}, $ Q^{\prime}(x ,\mu)\neq 0$, $\forall x \in (x_3,0)$ and thus $Q(x_3,\mu)=R(\mu)$. Finally, by \eqref{voiseta0}, we have $|R(\mu)-R(\mu_1)|<\epsilon$. Therefore, $R$ is continuous on $R^{-1}\left[(1,R_G)\right]$.
\end{proof}
\subsection{Proof of Lemma \ref{lemma2annexe}}
\begin{proof}
The uniqueness comes from the increase of $\phi$ described in Proposition \ref{nocut}. Furthermore, we cannot have $R(\mu_c)<R_G$. Indeed, the proof of Lemma \ref{continuityR}, implies that if $R(\mu_c)<R_G$ there exists $\mu_1$ such that $R(\mu_c)<R(\mu_1)<R_G$, which is in contradiction with the definition of $\mu_c$. 

We want now to prove that $Q^{\prime}(x_0,\mu_c)=0$ and that $R$ is continuous at $\mu_c$. Since the proof of these two points are very similar to that from Lemma $\ref{continuityR}$, we will skip some details. However, we cannot as in this lemma chose a domain which contains $(R_G,Q^{\prime}(x_0,\mu_c))$. Indeed, on such a domain, we would not have necessarily the Lipschitz condition because of the singularity of $G$ in $R_G$. We will thus take a slightly different one.  Let $\epsilon_1>0$.  We can take $x_1>x_0$ such that: $||X(x_1,\mu_c)-X(x_0,\mu_c)||_1<\epsilon_1$ and define $D$ by:
\begin{equation}
D=\bigcup_{x\in [0,x_1]} B(X(x,\mu_c),\epsilon_2 ),
\end{equation}
where $0<\epsilon_2<\epsilon_1$ is small enough to have that $d(D,\mathbb{R}\times \lbrace 0 \rbrace \cup \lbrace R_G \rbrace\times \mathbb{R})>0$. There exists $\eta>0$, such that for all $\mu$ satisfying $|\mu-\mu_c|<\eta$, we have $||X(x_1,\mu_c)-X(x_1,\mu)||_1<\epsilon_2$, and $X(x,\mu) \in D$, $\forall x \in [0,x_1]$. This implies in particular that for $|\mu-\mu_c|<\eta$, $R(\mu)\geq R_G-2 \epsilon_2$. Furthermore, we know that $R(\mu)\leq R_G$ and thus $R$ is continuous at $\mu_c$. \\
 Suppose now that we have $Q^{\prime}(x_0,\mu_c)<0$ and $\epsilon_1<|Q^{\prime}(x_0,\mu_c)|/4$. Let $s_1=Q(x_1,\mu_c)$ and choose $\mu$ and $s$ such that $|\mu_c-\mu|<\eta \wedge |\mu_c|$ and $s_1<s<R(\mu)$. Equation \eqref{eqa} yields:
\begin{align}
a^{\prime}(s,\mu)a(s,\mu)&=-2\mu a(s,\mu)-2\beta\left(G(s)-s\right) \nonumber\\
\frac{1}{2}(a^2(s,\mu)-a^2(s_1,\mu))&=-2\mu \int_{s_1}^s a(u,\mu)\mathrm{d}u-2\beta\int_{s_1}^s\left(G(u)-u\right)\mathrm{d}u \nonumber\\
|a(s,\mu)-a(s_1,\mu)|&\leq \frac{4}{|a(s_1,\mu)|} \left[ |\mu|(R_G-s_1)\sup_{x\in [x_0,0]}|Q^{\prime}(x,\mu_c)|\right.\nonumber\\
&\qquad \qquad \qquad \qquad \left.+\beta\int_{s_1}^{R_G}\left(G(u)-u\right)\mathrm{d}u\right]. \nonumber\\
|a(s,\mu)-a(s_1,\mu)|&\leq \frac{8}{|Q^{\prime}(x_0,\mu_c)|} \left[2 |\mu_c|(R_G-s_1)\sup_{x\in [x_0,0]}|Q^{\prime}(x,\mu_c)|\right.\nonumber\\
&\qquad \qquad \qquad \qquad \left.+\beta\int_{s_1}^{R_G}\left(G(u)-u\right)\mathrm{d}u\right].\label{derniereingermu}
\end{align}
Let $0<\epsilon<|Q^{\prime}(x_0,\mu_c)|/2$, by choosing $\epsilon_1$ small enough we have for all $\mu$ such that $|\mu_c-\mu|<\eta \wedge |\mu_c|$,  $|a(s_1,\mu)-a(R_G,\mu_c)|<\epsilon/2$ and thanks to \eqref{derniereingermu} that $|a(s,\mu)-a(s_1,\mu)|<\epsilon/2$ which implies that: $|a(s,\mu)|>|Q^{\prime}(x_0,\mu_c)|/2$ for all $s<R(\mu)$ and thus $R(\mu)=R_G$. If $\mu<\mu_c$ we then get a contradiction with the definition of $\mu_c$. Therefore $Q^{\prime}(x_0,\mu_c)=0$.
\end{proof}

\end{appendix}
\section*{Acknowledgements}
I am very grateful to my thesis advisor, Julien Berestycki, for his help throughout the writing of this article. 
 \bibliographystyle{plain}
 \bibliography{Biblioparticle}
\end{document}